\documentclass[a4paper, leqno,12pt]{amsart} 

\usepackage{amsmath}
\usepackage{amssymb}
\usepackage{graphicx}
\usepackage{tikz-cd}
\usepackage{hyperref}
\usepackage[all]{xy}

\usepackage{xspace}
\usepackage{bm}
\usepackage{amsmath}
\usepackage{amstext}
\usepackage{amsfonts}
\usepackage[mathscr]{euscript}
\usepackage{amscd}
\usepackage{latexsym}
\usepackage{amssymb}
\usepackage{enumerate}
\usepackage{xcolor}

\theoremstyle{plain}
\swapnumbers
    \newtheorem{theorem}[figure]{Theorem}
    \newtheorem{proposition}[figure]{Proposition}
    \newtheorem{lemma}[figure]{Lemma}
    
    \newtheorem{corollary}[figure]{Corollary}
    
    \newtheorem{subsec}[figure]{}
    
    \newtheorem*{thma}{Theorem A}
    \newtheorem*{thmb}{Theorem B}
    \newtheorem*{thmc}{Theorem C}

\theoremstyle{definition}
    \newtheorem{definition}[figure]{Definition}
    \newtheorem{example}[figure]{Example}

\theoremstyle{remark}

    \newtheorem{ack}[figure]{Acknowledgements}

\newcommand{\sect}{\setcounter{figure}{0}\section}

\newenvironment{mysubsection}[2][]
{\begin{subsec}\begin{upshape}\begin{bfseries}{#2.}
			\end{bfseries}{#1}}
		{\end{upshape}\end{subsec}}

\renewcommand{\thefigure}{\arabic{section}.\arabic{figure}}

\newenvironment{myeq}[1][]
{\stepcounter{figure}\begin{equation}\tag{\thefigure}{#1}}
{\end{equation}}

\newcommand{\mydiagram}[2][]
{\stepcounter{figure}\begin{equation}
     \tag{\thefigure}{#1}\vcenter{\xymatrix{#2}}\end{equation}}

\newcommand{\NN}{\mathbb{N}}
\newcommand{\A}{\mathfrak{A}}
\newcommand{\AAA}{\mathcal{A}}
\newcommand{\NA}{N\AAA}

\newcommand{\N}{N\A_B}
\renewcommand{\L}{L\A_B}
\newcommand{\Aut}{\operatorname{Aut}}
\newcommand{\AutB}{\Aut_B^B}

\newcommand{\map}{\operatorname{map}}
\newcommand{\mapbb}{ \map^B_B}
\newcommand{\Hom}{\operatorname{Hom}}
\renewcommand{\dim}{\operatorname{dim}}
\newcommand{\Id}{\operatorname{Id}}
\newcommand{\Ker}{\operatorname{Ker}}
\newcommand{\CW}{\operatorname{CW}}
\newcommand{\CWBB}{\CW_B^B}

\newcommand{\Xb}{X_b}
\newcommand{\Yb}{Y_b}
\newcommand{\fX}{f_X}
\newcommand{\fY}{f_Y}
\newcommand{\fb}{f_b}
\newcommand{\gb}{g_b}
\newcommand{\pX}{p_X}
\newcommand{\pY}{p_Y}
\newcommand{\fXb}{f_{\Xb}}
\newcommand{\fYb}{f_{\Yb}}
\newcommand{\XBY}{X\times _B Y}
\newcommand{\pBq}{p\times _B q}
\newcommand{\fBg}{f\times _B g}
\newcommand{\nprod}{{\bf X}}
\newcommand{\pprod}{p_1\times_B \ldots \times_B p_n}

\begin{document}
\title{MONOIDS RELATED TO SELF HOMOTOPY EQUIVALENCES OF FIBRED PRODUCT}
\author{Gopal Chandra Dutta}
\author{Debasis Sen}

\address{Department of Mathematics and Statistics\\
		Indian Institute of Technology, Kanpur\\ Uttar Pradesh 208016\\India}
	
\email{gopald@iitk.ac.in}
	
\address{Department of Mathematics and Statistics\\
		Indian Institute of Technology, Kanpur\\ Uttar Pradesh 208016\\India}
	
\email{debasis@iitk.ac.in}

\date{\today}
	
\subjclass[2010]{Primary 55R70, 55P10 ; \ Secondary: 55R05, 55Q05.}
\keywords{Self homotopy equivalences, monoids, reducibility, fibred product.}
	
\begin{abstract}
We introduce a nested sequence of monoids related to self homotopy equivalences of fibrewise pointed spaces, such that the limit is the group of homotopy classes of  fibrewise pointed self-equivalences. We explore this monoid for fibred product in terms of individual spaces. Further we study two related invariants associated to these monoids: self closeness number and self length. 
\end{abstract}
\maketitle

\maketitle
	
\sect{Introduction}\label{int}
The group of homotopy classes of self homotopy equivalences $\Aut(X)$ of a connected based space $X$ has been studied extensively for a long time (cf. \cite{agshe, rutter}). Various subgroups of $\Aut(X)$ have also been studied: e.g. the subgroup consists of the self-homotopy equivalences that induces the identity homomorphism on homotopy groups upto certain range (cf. \cite{sheihch,sheip,smiiahg}). Motivated by this, Choi and Lee introduced a sequence of monoids related to $\Aut(X)$ (see \cite{clcngh}). They defined $\AAA^k_{\#}(X) = \{f\in[X,X]~~:~~ \pi_i(f) \text{ is an isomorphism for } 0\leq i\leq k\}$. Clearly $\Aut(X) \subseteq \AAA^{k+1}_{\#}(X) \subseteq \AAA^{k}_{\#}(X).$ Related to this chain of inclusions, they also studied two invariants: self-closeness number and self length (see \cite{clslp, jlfcs}).

In this paper we introduce similar monoids for parametrized homotopy and study them for a fibred product in terms of the individual spaces. Let $B$ be a connected pointed CW-complex. A \emph{fibrewise pointed space} over $B$ consists of a space $X$ over $B$ with a section: $$B\xrightarrow{  s } X\xrightarrow{p} B,\text{               }p\circ s = \Id_B. $$ We denote it by the triple $(X,p,s)$ (\cite{cjfht}). Sometimes it is also called an \emph{ex-space} (cf. \cite{maypht}). A \emph{fibrewise pointed map} $f\colon X\to Y$ is a map such that $q\circ f= p$ and $f\circ s=t$. The collection of all such maps from $X$ to $Y$ is denoted by $\mapbb(X,Y)$. Two fibrewise pointed maps $f,g\colon X\to Y$ are called \emph{fibrewise pointed homotopic} if there exists a homotopy  $h_{\bar{t}}\colon X\to Y$ such that $h_0 = f, ~~ h_1 = g$ and $q\circ h_{\bar{t}} = p, ~~ h_{\bar{t}}\circ s = t, \forall ~~ \bar{t}\in I = [0,1]$ and written as $f\simeq^B _{B} g$. This is an equivalence relation on the set $\mapbb(X,Y)$ and the set of equivalence classes is denoted by $[X,Y]^B_B$. For notational convenience, we do not distinguish between a map $f\colon X\to Y$ and that of its homotopy class in $[X,Y]^B_B$.
The group of homotopy classes of fibrewise pointed self-equivalences of $X$ is denoted by $\AutB(X)$. This group has been studied for fibred product of spaces by several authors (cf. \cite{Gps, psfhe}). The group of homotopy classes of self equivalences for cartesian product ($B =\ast$) was also studied before (cf. \cite{ Hgsep, prshe, stshe})

 Let $\CWBB$ denote the full subcategory of fibrewise pointed spaces $(X,p,s)$ with $p\colon X\to B$ being a fibration,  the total space $X$ and fibre of $p$ have the homotopy type of connected CW-complexes. 
\begin{definition}\label{dmon}
For a space $X \in \CWBB$ and a non-negative integer $k$, we define $$\A^k_\#(X):=\{f\in [X,X]^B_B :~~~ \pi_{i}(f)\colon \pi_{i}(X)\xrightarrow{\cong} \pi_{i}(X), ~~0\leq i\leq k \}.$$
\end{definition}
\noindent If $B = \ast$ then $\A^k_{\#}(X) = \AAA^k_{\#}(X).$ Under the composition of maps the set $\A^k_\#(X)$ becomes a submonoid of $[X,X]^B_B$. If $m>n$, then $\A^m_{\#}(X)\subseteq \A^n_{\#}(X)$.  Therefore we have the following chain of inclusion, 
\begin{myeq}\label{eqchain}
\AutB(X)\subseteq \cdots \subseteq \A^{n+1}_{\#}(X)\subseteq \A^n_{\#}(X)\subseteq \cdots \subseteq \A^1_{\#}(X)\subseteq \A^0_{\#}(X) = [X,X]^B_B.
\end{myeq}
In the first part of this paper we study the monoid $\A^k_{\#}$ for fibred product. This is an extension of the results of Jun and Lee on factorisation of the monoid $\AAA^k(X\times Y)$ of product spaces (cf. \cite{jlfcs}). Recall that the fibred product of $(X, p, s)$ and $(Y,q, t)\in \CWBB$ is denoted by  $\XBY$ and is defined by the pull-back of the diagram 
$$
\xymatrix{
X \ar[rr]^{p} && B && Y\ar[ll]_{q}
}.
$$
We have a projection map to $B,$ given by $\pBq\colon \XBY\to B$, which is a fibration with section $(s,t)\colon B\to \XBY$. These maps are explicitly defined as $$(\pBq)(x,y):= p(x)=q(y) \text{ and } (s,t)(b):=(s(b),t(b)).
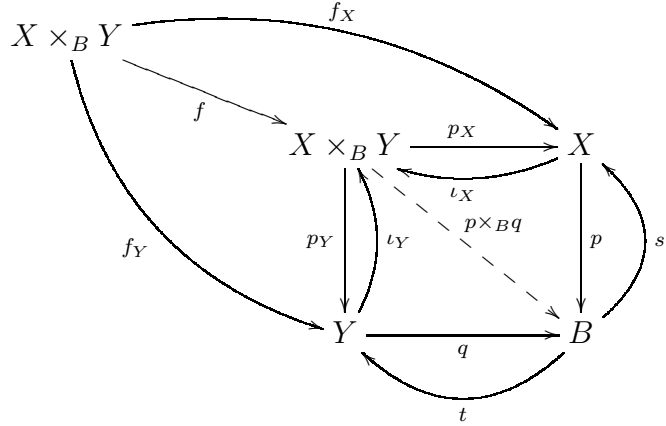
$$

$$
\xymatrix{
\XBY \ar@/^2.0pc/[rrrrd]^{\fX} \ar @/_2.0pc/[rrddd]_{\fY}\ar[rrd]_{f}  \\
&& \XBY \ar[dd]_{p_Y} \ar[rr]^{p_X} \ar@{-->}[rrdd]^{\pBq}  &&  X \ar[dd]^{p}  \ar@/^1pc/[ll]^{\iota_X} \\\\
&& Y \ar[rr]_{q} \ar@/_1pc/[uu]_{\iota_Y} && B \ar@/_2.0pc/[uu]_{s} \ar@/^2.0pc/[ll]^{t}
}
 $$
\noindent The section maps $s,t$ induce maps $$\iota_X\colon X\to \XBY,~~\iota_Y\colon Y\to \XBY,$$ which are defined as $$\iota_X(x):=(x,tp(x)),~~\iota_Y(y):=(sq(y),y).$$ Clearly $\iota_X, ~\iota_Y$ are sections of $p_X, ~ q_Y$ respectively. A self-map $f\colon \XBY\to \XBY$ is equivalent to a pair of maps $f = (\fX, \fY)$ where $\fX: \XBY\to X$ and $\fY: \XBY\to Y$ are maps in $\CWBB$. Given  $f =(\fX, \fY)\in\mapbb(\XBY, \XBY)$, then $(p_X,\fY), (\fX,p_Y)$ are self-maps of $\XBY$. 

 As in \cite{jlfcs}, we need the following notion of reducibility to state our results.
\begin{definition}
A map $f=(\fX,\fY)\in \A^k_{\#}(\XBY)$ is said to be \emph{$k$-reducible} if $(p_X,\fY),~~ (\fX,p_Y)\in \A^k_{\#}(\XBY)$.
\end{definition}
Recall that, in (cf. \cite[Section 1]{psfhe}), a map $f\in \AutB(\XBY)$ was defined to be reducible if $(\fX, p_Y)$ and $(p_X, \fY)$ both are in $\AutB(\XBY)$. Therefore a map $f\in \AutB(\XBY)$ is reducible if it is $k$-reducible for any non-negative integer $k$.  
We have two natural sub-monoids of  $\A^k_{\#}(\XBY$): 
\begin{myeq}\label{eq1}
\begin{split}
\A^k_{\#,X}(\XBY) & :=\{f\in \A^k_{\#}(\XBY): ~ p_X\circ f= p_X \}, \\
\A^k_{\#,Y}(\XBY) & :=\{f\in \A^k_{\#}(\XBY): ~ p_Y\circ f= p_Y \}.
\end{split}
\end{myeq}
We show that the product of this two sub-monoids is the whole of $\A^k_{\#}(\XBY)$. This can be thought as a monoid analogue of the result of Pave\v{s}i\'{c} \cite[Theorem 2.5]{psfhe} and a parametrized generalization of \cite[Theorem 1]{jlfcs}. We define a further submonoid, $$\A^{k,X}_{\#,Y}(\XBY):= \{(\fX,p_Y)\in \A^k_{\#,Y}(\XBY):~~ (\fX,p_Y)\circ \iota_X = \iota_X\}.$$ Similarly $\A^{k,Y}_{\#,X}(\XBY)$ is defined. By a split exact sequence of monoids, we mean usual definition in the case of groups forgetting the existence of inverse. 
 
\begin{thma}\label{thma}

Let $X,Y\in \CWBB.$ Then we have split exact sequences of monoids,
$$0\to\A^{k,X}_{\#,Y}(\XBY)\to\A^k_{\#,Y}(\XBY)\xrightarrow[]{}\A^k_{\#}(X)\to 0,$$ 
$$0\to\A^{k,Y}_{\#,X}(\XBY)\to\A^k_{\#,X}(\XBY)\xrightarrow[]{}\A^k_{\#}(Y)\to 0.$$ 

\noindent Moreover if each map $f=(\fX,\fY)\in \A^k_{\#}(\XBY)$ is $k$-reducible and satisfying the condition $\fY\simeq^B_B \fY\circ \iota_Y\circ p_Y$, then $$\A^k_{\#}(\XBY) = \A^k_{\#,X}(\XBY)\cdot \A^k_{\#,Y}(\XBY).$$
Similar result also holds for higher fibre product.
\end{thma}

\noindent (See Theorem \ref{ftsg}, Theorem \ref{mthnps}, Proposition \ref{splthp} for details.)

Booth and Heath obtained an exact sequence relating $\AutB(\XBY)$ and $\AutB(X)\times \AutB(Y)$ (\cite[Theorem 3.1]{Gps}). We obtain a similar result for the monoid $\A_{\#}.$

\begin{thmb}\label{thmb}
Assume that $X$ is a fibrewise group-like space such that $[X\wedge_B Y,X]^B_B = \{\ast\}$ and $[Y,X]^B_B = \{\ast\}$. Then we have the following short exact sequence of monoids,
$$0\to \A^{k,Y}_{\#,X}(X\times_B Y)\to \A^k_{\#}(X\times_B Y)\xrightarrow[]{} \A^k_{\#}(X)\times \A^k_{\#}(Y)\to 0.$$
Similar result also holds for higher fibred product.

\end{thmb}

\noindent (See Theorem \ref{hbthm}, Theorem \ref{hbfn} for details.)

Observe that for finite dimensional $X$, we have $\A^{\dim X}_{\#}(X) = \AutB(X) $. In general, $\lim\limits_{k\to\infty}\A^k_{\#}(X) = \AutB(X)$, by Whitehead's theorem and the fact that a homotopy equivalence between ex-spaces with the projection maps fibrations, is actually a fibre homotopy equivalence (cf. \cite{dpuf}). This motivates us to define parametrized \emph{self-closeness number} of $X\in \CWBB$ as the minimum non-negative integer $k$ such that  $\A^{k}_{\#}(X) = \AutB(X)$. We denote it by $\N(X)$. This is a generalisation of self-closeness number introduced by Choi, Lee as $\NA(X) := \min\big\{k: \AAA^k_{\#}(X) = \Aut(X), ~~k\geq 0\big\}$ (cf. \cite{clcngh}). They also introduced a related invariant \emph{self-length} of a space $X$ (\cite{clslp}). We define similar self-length in our parametrized setting as the number of strict inclusions in the chain of Equation \ref{eqchain}. We denote it by $\L(X)$. We obtain several properties of $\N(X)$ and $\L(X),$ and prove the following result.

\begin{thmc}
Let $X,Y\in \CWBB$. 
\begin{enumerate}[(i)]
\item  Then the self-closeness number satisfies $$ \N(\XBY) \geq \max \big\{ \N(X), \N(Y) \big\}.$$ In addition, the equality hold if for any non-negative integer $k$, each map in $\A^k(\XBY)$ is $k$-reducible.
 
\item For the self-length, we have 
 $$\L(X) + \L(Y) -  \#\big(I_{(X)}\cap I_{(Y)}\big)\leq \L(\XBY)$$
Moreover the equality holds if $Y$ is fibrewise group-like satisfying $[X\wedge_B Y,Y]^B_B = \{\ast\}$ and $[X,Y]^B_B = \{\ast\}$.
\end{enumerate}
\end{thmc}

\noindent (See Theorems \ref{scn}, \ref{escn}, \ref{eslps} for details.)

\begin{mysubsection}{Organisation}
The rest of the paper is organised as follows. In Section \ref{stwoproduct}, we consider the monoid $\A_{\#}^k$ for the fibred product of two spaces. We obtain several results on $k$-reducibility  to prove Theorem A and Theorem B. In Section \ref{ssclose}, we introduce self-closeness number of $X\in \CWBB$ and prove the first part of Theorem C. In Section \ref{sslength}, we introduce self-length, obtain its general properties and prove the second part of Theorem C. In the last Section \ref{shfb}, we generalise the results of the previous sections to higher fibred products. 

\begin{ack}
The first author would like to thank IIT Kanpur for Ph.D fellowship.
\end{ack}

\end{mysubsection}

\sect{$\A^k_{\#}(\XBY)$}\label{stwoproduct}

In this section we study the monoid $\A^k_{\#}$ (cf. Definition \ref{dmon}) for the fibred product $\XBY$ in terms of the individual spaces $X$ and $Y$. First we recall some basic definitions (see \cite{cjfht}). Let $(X,p,s)$ and $(Y, q, t)$ be fibrewise pointed spaces over $B$. 
 A fibrewise pointed map $f\colon X\to Y$ is called \emph{fibrewise constant} if $f = t\circ p$.  A fibrewise pointed map fibrewise homotopic to the fibrewise constant is called a \emph{fibrewise pointed null-homotopic}.
We now obtain some general results regarding the monoid $\A^k_{\#}(X)$. Let $(X, p, s) \in \CWBB$ and $X_b$ be the fibre of a point $b\in B$ under $p\colon X\to B$.  For $f\in \mapbb(X, X)$ let us denote the restriction to the fibre by 
\begin{myeq}\label{eqfiber}
X_b := p^{-1}(b),~~~~f_b:=f|_{X_b}:~X_b \to X_b.
\end{myeq}
Note that for  $f=(\fX,\fY)\in \mapbb(\XBY,\XBY),$
 $$(\XBY)_b = \Xb\times \Yb ~;~~~~\fb = (\fXb, \fYb): \Xb \times \Yb \to \Xb \times \Yb.$$ 

\begin{proposition}\label{iprop}
With notations as above, let $f\in \mapbb(X, X)$. Then $f\in \A^k_{\#}(X)$ if and only if $f_b\in \AAA^k_{\#}(X_b)$ for some $b\in B$. In particular, $f\in \AutB(X)$ if and only if $f_b \in \Aut(\Xb)$.
\end{proposition}
	
\begin{proof}
Consider the commutative diagram,  

\mydiagram[\label{dig1}]{
X_b \ar[dd]_{f_b}  \ar@{^{(}->}[rr]^{\iota}  && X \ar[dd]^{f} \ar[rr]^{p} && B \ar[dd]^{\Id_B} \\\\
X_b \ar@{^{(}->}[rr]_{\iota} && X \ar[rr]_{p} && B
}

Using the long exact sequence of homotopy groups, we have

\xymatrixrowsep{0.2in}
\xymatrixcolsep{0.1in}
\mydiagram[\label{dig}]{
	 \cdots \ar[rr]^{} &&\pi_{k+1}(B) \ar[dd]_{\Id_{\pi_{k+1}(B)}}  \ar[rr]^{} && \pi_k(X_b) \ar[dd]_{\pi_k(f_b)}  \ar[rr]^{\pi_k(\iota)}  && \pi_k(X) \ar[dd]^{\pi_k(f)} \ar[rr]^{\pi_k(p)} && \pi_k(B) \ar[dd]^{\Id_{\pi_k(B)}} \ar[rr]^{} && \pi_{k-1}(X_b) \ar[dd]^{\pi_{k-1}(f_b)} \ar[rr]^{\pi_{k-1}(\iota)} && \cdots \\\\
	 \cdots \ar[rr]^{} &&\pi_{k+1}(B) \ar[rr]_{}  && \pi_k(X_b) \ar[rr]_{\pi_k(\iota)} && \pi_k(X) \ar[rr]_{\pi_k(p)} && \pi_k(B) \ar[rr]_{} && \pi_{k-1}(X_b) \ar[rr]_{\pi_{k-1}(\iota)} && \cdots
}

\noindent If $f_b\in \A^k_{\#}(X_b)$, then using Five-Lemma we get $f\in \A^k_{\#}(X)$.

Conversely, let $f\in \A^k_{\#}(X)$. It is sufficient to prove that $\pi_k(f_b)$ is a bijective map. Let $a\in \pi_k(X_b)$ such that $\pi_k(f_b)(a) = 0$, then using commutativity $\pi_k(\iota)(a) = 0$, since $\pi_k(f)$ is an isomorphism. Thus $a\in \Ker(\pi_k(\iota))$, therefore there exists  $b'\in \pi_{k+1}(B)$ such that $b'\mapsto a$. Using commutativity we have $b'\mapsto 0$, consequently $a = 0$, thus $\pi_k(f_b)$ is injective. For surjectivity, let $\alpha\in \pi_k(X_b)$, then say $\alpha \mapsto \beta$ in $\pi_k(X)$ and $\beta \mapsto \gamma$ in $\pi_k(B)$. Clearly $\gamma = 0$. Since $\pi_k(f)$ is onto so there exists a $\hat{\beta}$ in $\pi_k(X)$ such that $\pi_k(f)(\hat{\beta}) = \beta$. Using commutativity we have $\pi_k(p)(\hat{\beta}) = 0$, i.e. $\hat{\beta}\in \Ker(\pi_k(p))$. So there exists $\hat{\alpha}\in \pi_k(X_b)$ such that $\hat{\alpha}\mapsto \hat{\beta}$ in $\pi_k(X)$. Let $\pi_k(f_b)(\hat{\alpha}) = \bar{a}$, by commutativity $\pi_k(\iota)(\bar{a}) = \beta$. Then $\pi_k(\iota)(\alpha - \bar{a}) = 0$, i.e. $\alpha - \bar{a}\in \Ker(\pi_k(\iota))$. Thus there exits $\bar{b}\in \pi_{k+1}(B)$ such that $\bar{b}\mapsto \alpha - \bar{a}$ in $\pi_k(\Xb)$. By commutativity $\pi_k(f_b)(\alpha - \bar{a}) = \alpha - \bar{a}$. Consequently, $\pi_k(f_b)(\hat{\alpha} + \alpha - \bar{a}) = \alpha$, i.e. $\pi_k(f_b)$ is surjective. 

\end{proof}
	
 Similar to \cite[Proposition 2.3]{psfhe}, we have the following. 

\begin{corollary}\label{kred}
A map $f\in \A^k_{\#}(\XBY)$ is $k$-reducible if and only if its restriction to a fibre $f_b\in \AAA^k_{\#}(X_b\times Y_b)$ is $k$-reducible, for some $b\in B$. 
\end{corollary}
	
\begin{proof}
Let $f=(\fX,\fY)\in \A^k_{\#}(\XBY)$ such that the restriction to a fibre $f_b = (f_{X_b},f_{Y_b})\in \AAA^k_{\#}(X_b\times Y_b)$ is $k$-reducible. Then by the definition of $k$-reducibility, we have $(p_{X_b},f_{Y_b}), (f_{X_b},p_{Y_b})\in \AAA^k_{\#}(X_b\times Y_b)$. Note that, $(p_X,f_Y)_b = (p_{X_b},f_{Y_b})$ and $(f_X,p_Y)_b = (f_{X_b},p_{Y_b})$. Consequently $(p_X,\fY),~~(\fX,p_Y)\in \A^k_{\#}(\XBY)$ by Proposition \ref{iprop}. So $f$ is $k$-reducible.
Converse can be proved similarly.
		
\end{proof}

We now obtain an equivalent criterion of $k$-reducibility. 

\begin{proposition}\label{mhom}
Let $f = (\fX,f_Y)\in \map^B_B(\XBY,\XBY)$. Then $(\fX,p_Y)\in \A^k_{\#}(\XBY)$ if and only if $\fX\circ \iota_X\in \A^k_{\#}(X)$ and $(p_X,\fY)\in \A^k_{\#}(\XBY)$ if and only if $\fY\circ \iota_Y\in \A^k_{\#}(Y)$. 

In particular, $(\fX,p_Y)\in \AutB(\XBY)$ if and only if $\fX\circ \iota_X\in \AutB(X)$ and $(p_X,\fY)\in \AutB(\XBY)$ if and only if $\fY\circ \iota_Y\in \AutB(Y)$.
\end{proposition}

\begin{proof}
From Proposition \ref{iprop} we say that $(\fX,p_Y)\in \A^k_{\#}(\XBY)$ if and only if $(\fX,p_Y)_b\in \AAA^k_{\#}(X_b\times Y_b)$. Observe that $(\fX,p_Y)_b = (\fX,p_Y)|_{\Xb\times \Yb} = (f_{X_b},p_{Y_b})$. Moreover using \cite[Lemma 2]{jlfcs} we have $(f_{X_b},p_{Y_b})\in \AAA^k_{\#}(X_b\times Y_b)$ if and only if $f_{X_b}\circ \iota_{X_b}\in \AAA^k_{\#}(X_b)$, where $\iota_{X_b}\colon X_b\to X_b\times Y_b$ be a map defined as $x\mapsto (x,t(b)),~~~\forall x\in X_b$. Observe that $(\fX\circ \iota_X)_b = (\fX\circ \iota_X)|_{\Xb} = f_{X_b}\circ \iota_{\Xb}$,
$$X\xrightarrow[]{\iota_X} \XBY\xrightarrow[]{\fX} X.$$ Therefore $(f_X,p_Y)\in \A^K_{\#}(\XBY)$ if and only if $(f_X\circ \iota_X)_b\in \AAA^k_{\#}(X_b)$. By using the Proposition \ref{iprop} we get the desired result.
The proof of the other part is similar.

\end{proof}

\begin{corollary}\label{ckredu}
A map $f = (f_X,f_Y)\in \A^k_{\#}(\XBY)$ is $k$-reducible if $f_X\circ \iota_X\in \A^k_{\#}(X)$ and  $f_Y\circ \iota_Y\in \A^k_{\#}(Y)$.

In particular, a map $f = (f_X,f_Y)\in \AutB(\XBY)$ is reducible if $f_X\circ \iota_X\in \AutB(X)$ and  $f_Y\circ \iota_Y\in \AutB(Y)$.
\end{corollary}

The following result gives some sufficient conditions for reducibility.
%
\begin{proposition}\label{rsem}
Let $X,Y\in \CWBB$ and $k$ be a non-negative integer.
\begin{enumerate}[(a)]
\item If $[X,Y]^B_B = \{\ast\}$ or $[Y,X]^B_B = \{\ast\}$, then each map $f \in  \A^k_{\#}(\XBY)$ is $k$-reducible. 
\item Assume that, for some $b\in B$, any one of the sets $\Hom(\pi_i(\Xb),\pi_i(\Yb))$ and $\Hom(\pi_i(\Yb),\pi_i(\Xb))$ is trivial for $1\leq i\leq k$. Then each map $f\in \A^k_{\#}(\XBY)$ is $k$-reducible.
\end{enumerate}
\end{proposition}

\begin{proof}
\begin{enumerate}[(a)]
\item Let $f = (\fX,\fY)\in \A^k_{\#}(\XBY)$. Then $(\fXb,\fYb) = (\fX,\fY)_b\in \AAA^k_{\#}(\Xb\times \Yb)$ from Proposition \ref{iprop}. So the induced homomorphism $\pi_i(\fXb,\fYb)$ is isomorphism for all $0\leq i\leq k$. Moreover using \cite[Section 2]{prshe}, we can say that  $\pi_i(\fXb,\fYb)$ is isomorphism  if and only if the following matrix is invertible:

\begin{myeq}\label{eqmat}
M_{i}(\fXb,\fYb):=
\begin{bmatrix}
	\pi_i(\fXb\circ \iota_{\Xb})& 	\pi_i(\fXb\circ \iota_{\Yb}) \\
 	\pi_i(\fYb\circ \iota_{\Xb})& \pi_i(\fYb\circ \iota_{\Yb})

\end{bmatrix}
\end{myeq}

\noindent If $[X,Y]^B_B = \{\ast\}$, then $\fY\circ \iota_X \simeq^B_B t\circ p$. Thus $(\fYb\circ \iota_{\Xb}) = (\fY\circ \iota_X)_b$ is homotopic to a constant map, i.e. $\pi_i(\fYb\circ \iota_{\Xb}) = 0$. Hence $\pi_i(\fXb\circ \iota_{\Xb}) ~~ \text{and}~~ \pi_i(\fYb\circ \iota_{\Yb})$ are isomorphisms. Therefore $\pi_i(\fXb\circ \iota_{\Xb}) ~~ \text{and}~~ \pi_i(\fYb\circ \iota_{\Yb})$ are isomorphisms for all $0\leq i\leq k$. So that, 
$$(\fX\circ \iota_X)_b = (\fXb\circ \iota_{\Xb})\in \AAA^k_{\#}(\Xb), ~~ (\fY\circ \iota_Y)_b = (\fYb\circ \iota_{\Yb})\in \AAA^k_{\#}(\Yb).$$ 
Moreover, using Proposition \ref{iprop} we have $\fX\circ \iota_X\in \A^k_{\#}(X),~~ \fY\circ \iota_Y\in \A^k_{\#}(Y)$. We will get similar result if $[Y,X]^B_B =\{\ast\}$.
\item Assume that $\Hom(\pi_i(\Xb),\pi_i(\Yb)) = \{\ast\}$ for some $b\in B$. Then for any map $f\in \A^k_{\#}(\XBY)$, we have $\pi_i(\fYb\circ \iota_{\Xb})\colon \pi_i(\Xb)\to \pi_i(\Yb)$ is trivial. Therefore using the matrix $M_i(\fXb,\fYb)$ of Equation \ref{eqmat}, we have $\pi_i(\fXb\circ \iota_{\Xb}) ~~ \text{and}~~ \pi_i(\fYb\circ \iota_{\Yb})$ are isomorphisms. We will get similar result if we assume $\Hom(\pi_i(\Yb),\pi_i(\Xb)) = \{\ast\}$. Therefore  $\pi_i(\fXb\circ \iota_{\Xb}) ~~ \text{and}~~ \pi_i(\fYb\circ \iota_{\Yb})$ are isomorphisms for all $0\leq i\leq k$. So that, 
$$(\fX\circ \iota_X)_b = (\fXb\circ \iota_{\Xb})\in \AAA^k_{\#}(\Xb), ~~ (\fY\circ \iota_Y)_b = (\fYb\circ \iota_{\Yb})\in \AAA^k_{\#}(\Yb).$$ 
Hence $\fX\circ \iota_X\in \A^k_{\#}(X),~~ \fY\circ \iota_Y\in \A^k_{\#}(Y)$ from Proposition \ref{iprop}.
\end{enumerate}
\end{proof}

\begin{example}\label{exredu}
Consider two ex-spaces $(B\times K(G,m),p,s)$ and $(B\times K(G,n),q,t)$ over $B$, where $p,q$ are first coordinate projection and $m\neq n$. For some $b\in B$, the fibre of $p$ and $q$ are $K(G,m)$ and $K(G,n)$ respectively. To show the reducibility of any map $f\in \A^k_{\#}\big(B\times K(G,m)\times_B B\times K(G,n)\big)$, its sufficient to check that either $\Hom \big(\pi_i\big(K(G,m)\big),\pi_i\big(K(G,n)\big)\big)$ or $\Hom \big(\pi_i\big(K(G,n)\big),\pi_i\big(K(G,m)\big)\big)$ is trivial by Proposition \ref{rsem}(b). Clearly, $\Hom \big(\pi_i\big(K(G,m)\big),\pi_i\big(K(G,n)\big)\big) = \{\ast\}$ for all $i$. Thus, for any $k\geq 0$, all maps in $\A^k_{\#}\big(B\times K(G,m)\times_B B\times K(G,n)\big)$ are $k$-reducible.
\end{example}

We now prove the first main result of this section. Recall from Equation \ref{eq1} that we have two natural sub-monoids $\A^k_{\#,X}(\XBY)$ and $\A^k_{\#,Y}(\XBY)$ of $\A^k_{\#}(\XBY)$. Clearly, the elements of $\A^k_{\#,X}(\XBY)$ are of the form $(p_X,\fY)$, where $\fY\circ \iota_Y\in \A^k_{\#}(Y)$. Similarly the elements of $\A^k_{\#,Y}(\XBY)$ are of the form $(\fX,p_Y)$, where $\fX\circ \iota_X\in \A^k_{\#}(X)$. We show that they fits into short exact sequences and then the product of them is the whole monoid.

\begin{theorem}\label{ftsg}
 Let $X,Y\in \CWBB.$ Then we have split exact sequences of monoids,
$$0\to\A^{k,X}_{\#,Y}(\XBY)\to\A^k_{\#,Y}(\XBY)\xrightarrow[]{\phi_X}\A^k_{\#}(X)\to 0,$$ 
$$0\to\A^{k,Y}_{\#,X}(\XBY)\to\A^k_{\#,X}(\XBY)\xrightarrow[]{\phi_Y}\A^k_{\#}(Y)\to 0.$$ 

\noindent Moreover if each map $f=(\fX,\fY)\in \A^k_{\#}(\XBY)$ is $k$-reducible and $\fY\simeq^B_B \fY\circ \iota_Y\circ p_Y$, then $$\A^k_{\#}(\XBY) = \A^k_{\#,X}(\XBY)\cdot \A^k_{\#,Y}(\XBY).$$
\end{theorem}
	
\begin{proof}
From Proposition \ref{mhom}, the map $\phi_X$ is well-defined. Let us observe that
for a self-map $g\colon X\to X$, we have $p_Y\circ \iota_X\circ g = p_Y\circ \iota_X.$ This is because, 
$$p_Y\circ \iota_X\circ g(x) = p_Y(g(x),tpg(x)) = tpg(x) = tp(x) = p_Y\circ \iota_X(x), ~~\forall x\in X.$$ 
Taking $(\fX, \pY) \in \A^k_{\#,Y}(\XBY)$
and $ g = \fX\circ \iota_X$, we have 
$$\iota_X\circ \fX\circ \iota_X = (p_X\circ \iota_X\circ \fX\circ \iota_X,p_Y\circ \iota_X\circ \fX\circ \iota_X) = (\fX\circ \iota_X,p_Y\circ \iota_X) = (\fX,p_Y)\circ \iota_X.$$
Therefore 
\begin{align*}
\phi_X\big((\fX,p_Y)\circ (h_X,p_Y)\big)& = \phi_X\big(\fX\circ(h_X,p_Y),p_Y\big)\\& = \fX\circ (h_X,p_Y)\circ \iota_X\\& = \fX\circ \iota_X\circ h_X\circ \iota_X\\& = \phi_X(\fX,p_Y)\circ \phi_X(h_X,p_Y).
\end{align*}
Hence $\phi_X$ is a homomorphism of monoids. We now identify $\Ker(\phi_X)$ with $\A^{k,X}_{\#,Y}(\XBY)$. Take $(\fX,p_Y)\in \Ker(\phi_X)$ then $\fX\circ \iota_X = \Id_X$. Therefore,
$$(\fX,p_Y)\circ \iota_X = (\fX\circ \iota_X, p_Y\circ \iota_X) = (\Id_X,p_Y\circ \iota_X) = (p_X\circ \iota_X,p_Y\circ \iota_X) = (p_X,p_Y)\circ \iota_X = \iota_X.$$
Thus $\Ker(\phi_X)\subseteq \A^{k,X}_{\#,Y}(\XBY)$. To show the other inclusion, 
take $(\fX,p_Y)\in \A^{k,X}_{\#,Y}(\XBY)$. Then $(\fX,p_Y)\circ \iota_X = \iota_X$. Pre-composing both sides with $p_X$ we obtain  $\fX\circ \iota_X = \Id_X.$ Hence $\Ker(\phi_X) = \A^{k,X}_{\#,Y}(\XBY).$
\noindent It remains to show that there exists a splitting map of $\phi_X$. Define
$$\psi_X\colon \A^k_{\#}(X)\to \A^k_{\#,Y}(X\times_BY), ~~~ h\mapsto (h\circ p_X,p_Y).$$ 
From Proposition \ref{mhom}, the map $\psi_X$ is well-defined. It is easy to verify that $\psi_X$ is a homomorphism such that $\phi_X\circ \psi_X = \Id_{\A^k_{\#}(X)}$.
Therefore we get the split exact sequence. Similarly we will obtain the other split exact sequence.

For the remaining part of the Theorem, observe that 
$$\A^k_{\#,X}(\XBY)\cap A^k_{\#,Y}(\XBY) = \{\ast \}.$$
Therefore if the factorization of an element in $\A^k_{\#}(\XBY)$ exists, it must be unique. Let $f=(\fX,\fY)\in \A^k_{\#}(\XBY)$. Then by $k$-reducibility we have $(p_X,\fY)\in \A^k_{\#,X}(\XBY)$ and $(\fX,p_Y)\A^k_{\#,Y}(\XBY)$. Then, 	
\begin{align*}
(p_X,\fY)\circ (\fX,p_Y)& = (\fX,\fY\circ(\fX,p_Y))\\
& \simeq^B_B (\fX,\fY\circ \iota_Y\circ p_Y\circ (\fX,p_Y))\\
& = (\fX,\fY\circ \iota_Y\circ p_Y)\\
& \simeq^B_B (\fX,\fY).
\end{align*}

\end{proof}

Using the next Lemma, we deduce a sufficient criterion for the conditions of Theorem \ref{ftsg} to be satisfied. Recall that the \emph{fibrewise wedge product}  $X\vee_B Y$ is defined as the pushout of the diagram $Y \xleftarrow{t} B \xrightarrow{s} X.$ Therefore $X\vee_B Y : = (X\times_B B) \cup (B\times_B Y)\subseteq X\times_B Y$. Moreover \emph{fibrewise smash product} is $X\wedge_B Y : = (X\times_B Y)/_B (X\vee_B Y)$.

$$
\xymatrix@R=1.2cm@C=1.2cm{
	B \ar[r]^{s}      \ar[d]_{t}         & X \ar[d]_{\iota_1} \ar@/^1pc/[rdd]^{\iota_X}   \\
	Y \ar[r]^{\iota_2} \ar@/_1pc/[rrd]_{\iota_Y} & X\vee_B Y \ar@{-->}[rd]^{\iota}             \\
	&                                   & X\times_B Y
}
$$

\vspace{0.5cm}

\noindent The maps $\iota_1\colon X\to X\vee_B Y$ and $\iota_2\colon Y\to X\vee_B Y$ can be described as $\iota_1(x):= (x,p(x)), ~~\iota_2(y):= (q(y),y)$. The map $\iota$ is defined by the universal property of pushout.

\begin{lemma}\label{hcm}
With notations as above, assume that $\iota^{\#}\colon [\XBY,Y]^B_B\to [X\vee_B Y,Y]^B_B$ is injective and $[X,Y]^B_B = \{\ast\}$. Then for each self-map $f\colon \XBY\to \XBY$, we have $\fY\simeq^B_B \fY\circ \iota_Y\circ p_Y$.
\end{lemma}
	
\begin{proof}
It is sufficient to show that $\iota^{\#}(\fY) = \iota^{\#}(\fY\circ \iota_Y\circ p_Y)$, i.e. $\fY\circ \iota = \fY\circ \iota_Y\circ p_Y\circ \iota$. To prove this, it is enough to show the equalities if we precompose with $\iota_1$ and $\iota_2.$ If we precompose with $\iota_1$, then both side becomes element of $[X,Y]^B_B = \{\ast\}$. Therefore they are homotopic.
Moreover for precomposing with $\iota_2$, observe that,
$$\fY\circ \iota_Y\circ p_Y\circ \iota\circ \iota_2 = \fY\circ \iota_Y\circ p_Y\circ \iota_Y  = \fY\circ \iota_Y =  \fY\circ \iota \circ \iota_2.$$
\end{proof}	

\noindent Recall that an ex-space $Y$ is called \emph{fibrewise group-like} if it is a fibrewise Hopf space for which the fibrewise multiplication is fibrewise homotopy-associative and admits a fibrewise homotopy inverse (cf. \cite[Part I, Section 11]{cjfht}. In this case $[Z, Y]^B_B$ is a group for any other ex-space $Z$.

\begin{corollary}\label{cftsg}
Assume $X,Y \in \CWBB$ and $Y$ is fibrewise group-like. If $[X\wedge_B Y,Y]^B_B = \{\ast\}$ and $[X,Y]^B_B = \{\ast\}$, then the conditions of Theorem \ref{ftsg} are satisfied.
\end{corollary}
	
\begin{proof}
We get the reducibility of maps from Lemma \ref{rsem}(a). For the other condition, consider the fibrewise pointed cofibre sequence: $$X\vee_B Y\xrightarrow[]{\iota}\XBY\xrightarrow[]{q} X\wedge_B Y.$$ Therefore we have the Barrat-Puppe sequence $$\cdots \to [\Sigma(X\vee_B Y),Y]^B_B\to [X\wedge_B Y,Y]^B_B\xrightarrow[]{q^{\#}}[\XBY,Y]^B_B\xrightarrow[]{\iota^{\#}}[X\vee_B Y,Y]^B_B.$$ Since $Y$ is a fibrewise group-like space, so the last three terms are groups and $\iota^{\#}$ and $q^{\#}$ are group homomorphisms. Hence $[X\wedge_B Y, Y]^B_B = \{\ast\}$ implies that $\iota^{\#}$ is injective. Therefore we get the desired result by using Lemma \ref{hcm}.

\end{proof}
\begin{example}
Let $X$ be a pointed CW-complex with base point $x_0$. Consider two ex-fibrations $(X\times B,p,s)$ and $(B,\Id,\Id)$, where $p\colon X\times B\to B$ is a trivial fibration with section $s\colon B\to X\times B$ defined as $b\mapsto (x_0,b)$. Observe that $[X\times B\wedge_B B, B]^B_B = \{\ast\}$ and $[X\times B, B]^B_B = \{\ast\}$. If $B$ is a fibrewise group-like space, then using Theorem \ref{ftsg} we have $$\A^k_{\#}(X\times B\times_B B) = \A^k_{\#,X\times B}(X\times B\times_B B)\cdot \A^k_{\#,B}(X\times B\times_B B).$$
\end{example}

The following Theorem is a monoid analogue of a result of Booth and Heath \cite[Theorem 3.1]{Gps}.
\begin{theorem}\label{hbthm}
Assume that $X$ is a fibrewise group-like space such that $[X\wedge_B Y,X]^B_B = \{\ast\}$ and $[Y,X]^B_B = \{\ast\}$. Then we have the following split exact sequence of monoids, 
$$0\to \A^{k,Y}_{\#,X}(X\times_B Y)\to \A^k_{\#}(X\times_B Y)\xrightarrow[]{\phi} \A^k_{\#}(X)\times \A^k_{\#}(Y)\to 0,$$
 where $\phi(\fX, \fY) = (f_X\circ \iota_X,f_Y\circ \iota_Y).$
\end{theorem}
\begin{proof}
Since $[Y,X]^B_B = \{\ast\}$, then each map $f\in \A^k_{\#}(X\times_B Y)$ is $k$-reducible from Proposition \ref{rsem}(a). So the map $\phi$ is well-defined. Moreover, from Corollary \ref{cftsg} we have $f_X \simeq^B_B f_X\circ \iota_X\circ p_X$, for each self-map $f = (f_X,f_Y)$.  As in the proof of the first part of Theorem \ref{ftsg} we have $\iota_Y\circ \fY\circ \iota_Y  = (p_X,\fY)\circ \iota_Y.$ Therefore,
\begin{align*}
\phi \big((f_X,f_Y)\circ (g_X,g_Y)\big)& = \phi \big(f_X\circ (g_X,g_Y),f_Y\circ (g_X,g_Y)\big)\\
& = \big(f_X\circ (g_X,g_Y)\circ \iota_X,f_Y\circ (g_X,g_Y)\circ \iota_Y \big)\\
& \simeq^B_B \big(f_X\circ \iota_X\circ p_X\circ (g_X,g_Y)\circ \iota_X,f_Y\circ(g_X\circ \iota_X\circ p_X,g_Y)\circ \iota_Y \big)\\
& = \big(f_X\circ \iota_x\circ g_X\circ \iota_X,f_Y\circ (g_X\circ \iota_X\circ p_X\circ \iota_Y,g_Y\circ \iota_Y) \big)\\
& = \big(f_X\circ \iota_X\circ g_X\circ \iota_X,f_Y\circ(p_X\circ \iota_Y,g_Y\circ \iota_Y) \big)\\
& = \big(f_X\circ \iota_X\circ g_X\circ \iota_X,f_Y\circ (p_X,g_Y)\circ \iota_Y \big)\\
& = \big(f_X\circ \iota_X\circ g_X\circ \iota_X,f_Y\circ \iota_Y\circ g_Y\circ \iota_Y \big)\\
& = \phi(f_X,f_Y)\circ \phi(g_X,g_Y).
\end{align*}
This shows that $\phi$ is a monoid homomorphism. 

Any element of $\A^{k,X}_{\#,Y}(\XBY)$ is of the form $(f_X,p_Y)$ satisfying $(f_X,p_Y)\circ \iota_X = \iota_X$. Precomposing with $p_X$ we get $f_X\circ \iota_X = \Id_X$. Now post-composing with $p_X$ we have $f_X\circ \iota_X\circ p_X = p_X$. This implies $f_X\simeq^B_B  p_X$.
Therefore $\A^{k,X}_{\#,Y}(\XBY) = \{\ast\}$. So $\Ker(\phi)$ can be identified with $\Ker(\phi_Y) = \A^{k,Y}_{\#,X}(X\times_B Y)$. Let us define a map $$\psi\colon \A^k_{\#}(X)\times \A^K_{\#}(Y)\to \A^k_{\#}(X\times_B Y),~~\psi(f,g) := (f\circ p_X,g\circ p_Y).$$ 
Note that, $$(f\circ p_X,g\circ p_Y) = (f\circ p_X,p_Y)\circ(p_X,g\circ p_Y)\in \A^k_{\#}(\XBY),$$ since  $(f\circ p_X,p_Y),~(p_X,g\circ p_Y)\in \A^k_{\#}(\XBY)$ from Proposition \ref{mhom}. Therefore $\psi$ is well-defined.
Clearly $\psi$ is a monoid homomorphism such that $\phi \circ \psi = \Id_{\A^k_{\#}(X)\times \A^k_{\#}(Y)}$. This gives a splitting map of $\phi$.

\end{proof}	

\sect{self-closeness number}\label{ssclose}
The self-closeness number of a space was introduced by Choi and Lee in \cite{clcngh}. Some computations of this invariant were further done in \cite{shec,OYscf,OYsc}. We define a parametrized version of self-closeness number. 
Recall that we have the following chain: 
$\AutB(X)\subseteq \cdots \subseteq \A^{n+1}_{\#}(X)\subseteq \A^n_{\#}(X)\subseteq \cdots \subseteq \A^1_{\#}(X)\subseteq \A^0_{\#}(X) = [X,X]^B_B.$ For a finite dimensional $X$, we have $\A^{\dim X}_{\#}(X) = \AutB(X) $. In general, if $k = \infty $ then $\A^k_{\#}(X) = \AutB(X)$. 

\begin{definition}
The \emph{self-closeness number} of $X\in \CWBB$ is defined as $$\N(X) := \min \big\{ k: \A^k_{\#}(X) = \AutB(X),~~  k\geq 0 \big\}.$$
If no such $k$ exists, then $\N(X)$ is defined to be infinite. 
\end{definition}

\begin{lemma}
Let $X\in \CWBB$ is $m$-connected and of dimension $n$ where $m<n$. Then $m+1\leq \N(X)\leq n$, provided that $\AutB(X)\neq [X,X]^B_B$.
\end{lemma}

\begin{proof}
Since $X$ is $m$ connected, we have $\pi_i(X) = 0,~~~\forall~ i\leq m$. Therefore $$\A^{m}_{\#}(X) = \cdots = \A^1_{\#}(X) = \A^0_{\#}(X) = [X,X]^B_B.$$	Since $\AutB(X) \neq [X,X]^B_B$, thus $\N(X)\geq m+1$. To complete the proof, it is sufficient to show that $\A^n_{\#}(X) = \AutB(X)$. Let $f\in \A^n_{\#}(X)$, then the map $\pi_i(f)\colon \pi_i(X)\to \pi_i(X)$ is isomorphism for all $\forall~ i\leq n$. From \cite[Proposition 17.6]{cjfht}, we have $f_*\colon [X,X]^B_B\to [X,X]^B_B$ is a surjective. Therefore there is a map $h\in [X,X]^B_B$ such that $f_*(h) = \Id_X,~~ \text{i.e.}~~ f\circ h \simeq^B_B \Id_X$. So $ \pi_i(f)\circ \pi_i(h) = \Id_{\pi_i(X)}.$ Thus $\pi_i(h)$ is an isomorphism for all $i\leq n$. Therefore $h_*\colon [X,X]^B_B\to [X,X]^B_B$ is surjective, so there is a map $g\in [X,X]^B_B$ such that $h_*(g) = \Id_X,~ \text{i.e.}~ h\circ g\simeq^B_B \Id_X$. Thus $f\circ h\circ g \simeq^B_B f\Rightarrow g\simeq^B_B f$. Therefore $f\circ h \simeq^B_B h\circ f\simeq^B_B \Id_X$, consequently $f\in \AutB(X)$.

\end{proof}

The following Corollary give us some condition when $\N(X)$ is equal to an integer $n$.
\begin{corollary}
Let $X\in \CWBB$ is $n-1$ connected and of dimension $n$. If $\AutB(X) \neq [X,X]^B_B$, then $\N(X) = n$.
\end{corollary}

We now show that $\N(X)$ is homotopy invariant. 

\begin{proposition}\label{eqscn}
If $X,Y\in \CWBB$ are fibrewise pointed homotopy equivalent, then $\N(X) = \N(Y)$.
\end{proposition}

\begin{proof}
It is sufficient to show that given any non-negative integer $k, ~\AutB(X) = \A^k_{\#}(X)$ if and only if $\AutB(Y) = \A^k_{\#}(Y)$. Let $\phi\colon X\to Y$ be a fibrewise pointed homotopy equivalence with homotopy inverse $\psi\colon Y\to X$.
Then $\phi_*\colon \A^k_{\#}(X)\to \A^k_{\#}(Y)$ defined as $f \mapsto \phi\circ f\circ \psi$, is a monoid isomorphism with inverse $\psi_*\colon \A^k_{\#}(Y)\to \A^k_{\#}(X)$. 
Therefore the restriction map $\phi_*\colon \AutB(X)\to \AutB(Y)$ is an isomorphism of group.
Let us assume that $\AutB(X) = \A^k_{\#}(X)$. If $h\in \A^k_{\#}(Y)$ then $\psi_*(h)\in \A^k_{\#}(X) = \AutB(X)$. Therefore $\phi_*\big(\psi_*(h)\big)\in \AutB(Y)$, i.e. $h\in \AutB(Y)$. Thus $\A^k_{\#}(Y)\subseteq \AutB(Y)$, i.e. $\A^k_{\#}(Y) = \AutB(Y)$. Similarly we can prove that if $\AutB(Y) = \A^k_{\#}(Y)$ then $\AutB(X) = \A^k_{\#}(X)$.

\end{proof}

\begin{lemma}\label{rbtsn}
Let $(X,p,s)\in \CWBB$, where $X_b\xrightarrow[]{\iota} X\xrightarrow[]{p} B$. Then $$\N(X)\leq \NA(X_b).$$

If $p$ is a trivial fibration then the equality holds.
\end{lemma}

\begin{proof}
We have to consider only the case $\NA(X_b)< \infty$. Assume that $\NA(X_b) = m$. Let $f\in \A^m_{\#}(X)$. From Proposition \ref{iprop} we have $f_b\in \AAA^m_{\#}(X_b) = \Aut(X_b)$. Therefore using the Proposition \ref{iprop} again we have $f\in \AutB(X)$. So $\A^m_{\#}(X) = \AutB(X)$. Hence by the definition of self-closeness number we have $\N(X)\leq m = \NA(X)$.

If $p$ is a trivial fibration, then $X=X_b\times B$ with $p$ is a second coordinate projection. If possible let $\A^k_{\#}(X) = \AutB(X)$ where $k<m$. Then for any element $g\in \AAA^k_{\#}(X_b)$ we have $g\times \Id_B\in \A^k_{\#}(X)= \AutB(X)$. 
Using Proposition \ref{iprop} we have  $(g\times \Id_B)_b = g\in \Aut(X_b)$. Thus $\AAA^k_{\#}(X_b) = \Aut(X_b)$, which contradict the fact that $\NA(X_b) = m$. Hence from the definition of self-closeness number we get  $\N(X) = m =\NA(X_b)$.

\end{proof}

Having established some basic properties of parametrized self-closeness number, we now find this number of fibred product in terms of the individual spaces.

\begin{theorem}\label{scn}
Let $X,Y\in \CWBB$, then we have $$ \N(\XBY) \geq \max \big\{ \N(X), \N(Y) \big\}.$$ 
\end{theorem}

\begin{proof}
Consider the case $\N(\XBY)< \infty$. Let us assume that $\N(\XBY) = n$. It is sufficient to show that for any $f\in \A^n_{\#}(X), ~~ g\in \A^n_{\#}(Y)$ implies $f\in \AutB(X)$ and $g\in \AutB(Y)$. From Proposition \ref{iprop} we have $\fb\in \AAA^n_{\#}(\Xb)$ and        $ g_b\in \AAA^n_{\#}(\Yb)$ for some $b\in B$. Therefore $\pi_i(\fb)$ and $\pi_i(\gb)$ are isomorphisms for all $i = 0,\cdots,n$. So, $\pi_i(f_b\times g_b)$ are isomorphisms for $i = 0,\cdots,n$ i.e. $f_b\times g_b\in \AAA^n_{\#}(\Xb\times \Yb)$. 
 Now, $ \fBg\colon \XBY\to \XBY$ is a fibrewise pointed map defined by $(\fBg)(x,y):= \big(f(x),g(y)\big)$. Moreover $(\fBg)_b = (\fBg)|_{\Xb\times \Yb} = (\fb\times \gb)$. Therefore from Proposition \ref{iprop} we have $\fBg\in \A^n_{\#}(\XBY) = \AutB(\XBY).$ Again using Proposition \ref{iprop} we have $(f_b\times g_b)\in \Aut(\Xb\times \Yb)$.
Thus $\pi_i(\fb\times \gb)$ is an isomorphisms for all nonnegative integers $i$. So the maps $\pi_i(\fb)\times \pi_i(\gb)\colon \pi_i(\Xb)\times \pi_i(\Yb)\to \pi_i(\Xb)\times \pi_i(\Yb)$ are isomorphisms for all $i\geq 0$. Therefore $\pi_i(\fb)$ and $\pi_i(\gb)$ are isomorphisms for all $i\geq 0$. So $f_b\in \Aut(\Xb)$ and $g_b\in \Aut(\Yb)$. Hence $f\in \AutB(X),~~ g\in \AutB(Y)$ from Proposition \ref{iprop}. This implies that $\N(X)\leq n$ and $ \N(Y)\leq n$. Hence we have the desired result.

\end{proof}

The following result shows that if we assume the condition of reducibility, then we have equality in the above Theorem \ref{scn}. 

\begin{theorem}\label{escn}
Let $X,Y\in \CWBB$. For any $k\geq 0$, assume that each map of $\A^k_{\#}(\XBY)$ is $k$-reducible. Then $$\N(\XBY) = \max \big\{\N(X),\N(Y)\big\}.$$

\end{theorem}

\begin{proof}
If $\N(X)$ or $\N(Y)$ is infinite then $\N(\XBY)$ is infinite from Theorem \ref{scn}. So we only need to consider the case $\N(X)< \infty$ and $\N(Y)< \infty$. Let $\N(X) = m,~\N(Y) = n$ and $k =\max\{m,n\}$. Take $f = (\fX,\fY)\in \A^k_{\#}(\XBY)$. Then $\fX\circ \iota_X\in \A^k_{\#}(X)$ and $\fY\circ \iota_Y\in \A^k_{\#}(Y)$, using Corollary \ref{ckredu}. Since $k = \max \{m,n\}$, therefore $\fX\circ \iota_X\in \A^k_{\#}(X)\subseteq \A^m_{\#}(X) = \AutB(X)$ and  $\fY\circ \iota_Y\in \A^k_{\#}(Y)\subseteq \A^n_{\#}(Y) = \AutB(Y)$. Hence by Proposition \ref{mhom}, we have $$(\fX,\pY)\in \AutB(\XBY), ~(\pX,\fY)\in \AutB(\XBY).$$
As in the proof of \cite[Theorem 2.5]{pshp}  we have, $$f = (\fX,\fY) = \big(\pX,\fY\circ (\fX,\pY)^{-1}\big)\circ (\fX,\pY).$$
Note that, $(\fX,\pY)^{-1} = (\bar{f}_X,p_Y)\in \AutB(\XBY)$ ~(cf. \cite[Proposition 2.3(a)]{pshp}).
Observe that $$(p_X,\fY)\circ (\fX,p_Y)^{-1} = (\bar{f}_X,\fY\circ (\fX,p_Y)^{-1})\in \AutB(\XBY).$$
By reducibility assumption we get $\fY\circ (\fX,p_Y)^{-1}\circ \iota_Y\in \AutB(Y)$. Therefore $(p_X,\fY\circ (\fX,p_Y)^{-1})\in \AutB(\XBY)$ from Proposition \ref{mhom}. Consequently $f\in \AutB(\XBY)$.
Therefore $\A^k_{\#}(\XBY) = \AutB(\XBY)$, so $\N(\XBY)\leq k$. Using Theorem \ref{scn} we get $$\N(\XBY) = \max\big \{\N(X),\N(Y)\}.$$

\end{proof}

Combining Proposition \ref{rsem}(a) and Theorem \ref{escn} we obtain the following Corollary.
\begin{corollary}
Let $X,Y\in \CWBB$ such that $[X,Y]^B_B = \{\ast\}~~(\text{ or } [Y,X]^B_B = \{\ast\})$, then  $$\N(\XBY) = \max \big\{\N(X),\N(Y)\big\}.$$
\end{corollary}

\begin{example}
	
Let $(X,p,s)$ and $(B\times K(G,n),q,t)\in \CWBB$ are two ex-spaces over $B$, where $q\colon B\times K(G,n)\to B$ is first coordinate projection. Observe that  $X\times_B(B\times K(G,n))\cong X\times K(G,n)$. Therefore $\N\big(X\times_B (B\times K(G,n))\big) = \N\big(X\times K(G,n)\big)$. Moreover, if $\pi_i(X) = 0$ for all $i>n$ then $\pi_i\big(X\times K(G,n)\big) = 0$ for all $i>n$.
Therefore $\A_{\#}^n(X\times K(G,n)) = \AutB(X\times K(G,n)$. Hence, 
$$\N\big(X\times_B (B\times K(G,n))\big) = \N\big(X\times K(G,n)\big) \leq n.$$
	
\end{example}

\begin{example}\label{exam2sc}
	
Consider two ex-spaces $\big(K(G,m)\times B,p,s\big)$ and $\big(M(G,n)\times B,q,t\big)$ over $B$, where $G$ is an abelian group with $n\geq 3$ and $m<n$. Observe that self-maps of $K(G,m)\times M(G,n)$ are reducible (see \cite[Corollary 2.2]{pshp}). Therefore each self-map of $\big(K(G,m)\times B\big)\times_B\big(M(G,n)\times B\big)$ is reducible using Corollary \ref{kred}. Moreover $\big(K(G,m)\times B\big)\times_B\big(M(G,n)\times B\big)\cong K(G,m)\times M(G,n)\times B$. From Theorem \ref{escn}, we have  
\begin{align*}
\N\big(K(G,m)\times M(G,n)\times B\big) = &~\N\big(\big(K(G,m)\times B\big)\times_B\big(M(G,n)\times B\big)\big)\\
=&~ \max\big\{\N(K(G,m)\times B),\N(M(G,n)\times B)\big\}.
\end{align*}
From Lemma \ref{rbtsn} we have $\N\big(K(G,m)\times B\big) = \NA\big(K(G,m)\big) = m$. Similarly  
$\N\big(M(G,n)\times B \big) = \NA\big(M(G,n)\big) = n$, (see \cite[Corollary 3]{clcngh}). Consequently $$\N\big(K(G,m)\times M(G,n)\times B\big) = n.$$
\end{example}

\sect{Self-length}\label{sslength}

 In the previous section we studied self-closeness number, which is the minimum value of $n$ such that $\A^n_{\#}(X) = \AutB(X)$. Here we calculate the number of strict inclusions in the chain of Equation \ref{eqchain}, which was studied by Choi and Lee in \cite{clslp}.

\begin{definition}
The number of strict inclusions in the chain of Equation \ref{eqchain} is called the \emph{self-length} of $X$ and is denoted by $\L(X)$.
\end{definition}

We define a set $I_{(X)}$ of non-negative integers by, $$I_{(X)} := \big\{n\in \NN\cup \{0\}~~:~~ \A^{n+1}_{\#}(X)\subsetneq \A^n_{\#}(X)\big\}.$$
Clearly the cardinality of the set $$\# I_{(X)} = \L(X).$$

We now explore the relation between the self closeness number and the self length. 
\begin{proposition}\label{rbslsc}
For any $X\in \CWBB$, we have $\L(X)\leq \N(X)$. In particular, self-length of $X$ is finite if and only if $X$ has finite self-closeness number.
\end{proposition}

\begin{proof}
Consider the case $\L(X)<\infty$. Let $\L(X) = m$. Then we must have $$\AutB(X) = \cdots = \A^{k}_{\#}(X)\subsetneq \A^{k-1}_{\#}(X) \subseteq \cdots \subseteq \A^{n}_{\#}(X)\subseteq \cdots \subseteq \A^0_{\#}(X) = [X,X]^B_B,$$ for some non-negative integer $k$. Clearly $k\geq m$, otherwise $\L(X)<m$. By the definition of self-closeness number, we have $\N(X) = k$. Thus we get the desired result.

\end{proof}

\begin{corollary}
\begin{enumerate}[(a)]
\item For any finite dimensional $X\in \CWBB$, we have $\L(X)< \infty$.
\item For any $X\in \CWBB$, if $\N(X) = n$ and $X$ is $m$-connected with $m<n$, then $1\leq \L(X)\leq n-m$. 
\end{enumerate}
\end{corollary}

\begin{proof}
(a) If $\dim(X) = n$, then $\N(X)\leq n$. Therefore $\L(X)\leq \N(X)\leq n$.

(b) If $X$ is $m$-connected, then $\pi_i(X) = 0$ for all $i\leq m$. Thus observed that $\A^m_{\#}(X) = \A^{m-1}_{\#}(X) = \cdots =\A^0_{\#}(X)= [X,X]^B_B$. Since $\N(X) = n$, therefore we have the chain $$\AutB(X) = \A^n_{\#}(X)\subsetneq \A^{n-1}_{\#}(X)\subseteq \cdots \subseteq \A^m_{\#}(X) = [X,X]^B_B.$$ Therefore there are at most $n-m$ strict inclusions. Hence $1\leq \L(X)\leq n-m$.

\end{proof}

As for self closeness number, the self-length is also homotopy invariant. 

\begin{theorem}
Assume that $X,Y\in \CWBB$ are fibrewise pointed homotopy equivalent. Then $\L(X) = \L(Y)$. 
\end{theorem}

\begin{proof}
It is sufficient to show that for any non-negative integer $k,~~\A^{k+1}_{\#}(X)\subsetneq \A^k_{\#}(X)$ if and only if $\A^{k+1}_{\#}(Y)\subsetneq \A^k_{\#}(Y)$ which is equivalent to show that $\A^{k+1}_{\#}(X) = \A^k_{\#}(X)$ if and only if $\A^{k+1}_{\#}(Y) = \A^k_{\#}(Y)$. As in the proof of Proposition \ref{eqscn}, we have monoid isomorphism $\phi_*\colon \A^k_{\#}(X)\to \A^k_{\#}(Y)$ with inverse $\psi_*\colon \A^k_{\#}(Y)\to \A^k_{\#}(X)$. Consider the commutative diagram 

$$
\xymatrix{
	\A^{k+1}_{\#}(X) \ar[dd]_{\phi_*}  \ar[rr]^{\subseteq}  &&  \A^k_{\#}(X) \ar[dd]^{\phi_*} \\\\
	\A^{k+1}_{\#}(Y) \ar[rr]^{\subseteq} &&  \A^k_{\#}(Y)
}
$$
From the commutativity we say that the above horizontal map is identity if and only if the lower horizontal map is identity. Therefore $\A^{k+1}_{\#}(X) = \A^k_{\#}(X)$ if and only if $\A^{k+1}_{\#}(Y) = \A^k_{\#}(Y)$.

\end{proof}

Under certain conditions, the following proposition gives an equivalent criterion for a strict inequality of the chain (see Equation \ref{eqchain}) for the case of fibred product. 

\begin{proposition}\label{slps}
Let $X,Y\in \CWBB$. For any non-negative integer $k$, we have
\begin{enumerate}[(a)]
\item If $\A^{k+1}_{\#}(X)\subsetneq \A^k_{\#}(X)$  (or $\A^{k+1}_{\#}(Y)\subsetneq \A^k_{\#}(Y)$) then $\A^{k+1}_{\#}(\XBY)\subsetneq \A^k_{\#}(\XBY)$.
\item The converse holds if $Y$ is a fibrewise group-like space that satisfies the conditions $[X\wedge_B Y,Y]^B_B = \{\ast\}$ and $[X,Y]^B_B = \{\ast\}$. 
\end{enumerate}
\end{proposition}

\begin{proof}
(a)
Let $\A^{k+1}_{\#}(X)\subsetneq \A^k_{\#}(X)$. If possible, assume that $\A^{k+1}_{\#}(\XBY) = \A^k_{\#}(\XBY)$. Choose $f\in \A^k_{\#}(X) \setminus \A^{k+1}_{\#}(X)$. Then for any $g\in \A^k_{\#}(Y)$ we have 
$$f\times_B g\in \A^k_{\#}(\XBY) = \A^{k+1}_{\#}(\XBY).$$
 Then $f_b\times g_b = (f\times_B g)_b\in \A^{k+1}_{\#}(X_b\times Y_b)$ for some $b\in B$, from Proposition \ref{iprop}. Thus $\pi_i(f_b\times g_b)\colon \pi_i(X_b\times Y_b)\to \pi_i(X_b\times Y_b)$ is an isomorphism for all $i\leq k+1$. So $\pi_i(f_b)\colon \pi_i(X_b)\to \pi_i(X_b)$ is an isomorphism for all $i\leq k+1$. Hence $f_b\in \A^{k+1}_{\#}(X_b)$, so that $f\in \A^{k+1}_{\#}(X)$ from Proposition \ref{iprop}. This is a contradiction. We will get the similar result if we take $\A^{k+1}_{\#}(Y)\subsetneq \A^k_{\#}(Y)$.

(b)
For the converse implication, it is given that $\A^{k+1}_{\#}(\XBY)\subsetneq \A^k_{\#}(\XBY)$. If possible assume that $\A^{k+1}_{\#}(X) = \A^k_{\#}(X)$ and $\A^{k+1}_{\#}(Y) = \A^k_{\#}(Y)$. Since $\A^{k+1}_{\#}(\XBY)\subsetneq \A^k_{\#}(\XBY)$, let $h = (h_X,h_Y)\in \A^k_{\#}(\XBY) \setminus \A^{k+1}_{\#}(\XBY)$. From Corollary \ref{ckredu} and Proposition \ref{rsem}(a) we have $h_X\circ \iota_X\in \A^k_{\#}(X)$ and  $h_Y\circ \iota_Y\in \A^k_{\#}(Y)$. Thus $$h_X\circ \iota_X\in \A^{k+1}_{\#}(X) = \A^k_{\#}(X),~~h_Y\circ \iota_Y\in \A^{k+1}_{\#}(Y) = \A^k_{\#}(Y).$$ Therefore $(p_X,h_Y)$ and $(h_X,p_Y)\in \A^{k+1}_{\#}(\XBY)$ from Proposition \ref{mhom}. Combining Theorem \ref{ftsg} and Corollary \ref{cftsg} we have $h = (h_X,h_Y) = (p_X,h_Y)\circ (h_X,p_Y)\in \A^{k+1}_{\#}(\XBY)$, which is a contradiction. Consequently at least one of $\A^{k+1}_{\#}(X)\subsetneq \A^k_{\#}(X)$ and $\A^{k+1}_{\#}(Y)\subsetneq \A^k_{\#}(Y)$ holds.

\end{proof}

We use the above proposition to deduce a relation of self-length of fibred product in terms of the individual spaces.

\begin{theorem}\label{eslps}
Let $X,Y\in \CWBB$. Then
$$\L(X) + \L(Y) -  \#\big(I_{(X)}\cap I_{(Y)}\big)\leq \L(\XBY)$$
Moreover the equality holds, if $Y$ is a fibrewise group-like space that satisfies the conditions $[X\wedge_B Y,Y]^B_B = \{\ast\}$ and $[X,Y]^B_B = \{\ast\}$.
\end{theorem}

\begin{proof}
Note that if $\L(X) (\text{ or } \L(Y))$ is infinite then $\L(\XBY)$ is  infinite. So we have to only consider the case $\L(X)<\infty \text{ and } \L(Y)<\infty$.  Observe that $k\in I_{(X)} \cup I_{(Y)}$ implies $\A^{k+1}_{\#}(\XBY)\subsetneq \A^k_{\#}(\XBY)$ from Proposition \ref{slps}(a). Therefore the number of strict inclusions of the chain $\AutB(\XBY) \subseteq \cdots \subseteq \A^n_{\#}(\XBY)\subseteq \cdots \subseteq \A^0_{\#}(\XBY) = [\XBY,\XBY]^B_B$ is at least the cardinality of the set $I_{(X)}\cup I_{(Y)}$. Note that, $$\#\big(I_{(X)}\cup I_{(Y)}\big) = \#I_{(X)}  + \# I_{(X)} - \#\big(I_{(X)}\cap I_{(Y)}\big) = \L(X) + \L(Y) - \#\big(I_{(X)}\cap I_{(Y)}\big).$$

For the last implication of the proof, see that $\A^{k+1}_{\#}(\XBY)\subsetneq \A^k_{\#}(\XBY)$ if and only if $k\in I_{(X)} \cup I_{(Y)}$, from Proposition \ref{slps}(b).  Therefore the number of strict inclusions of the chain $\AutB(\XBY) \subseteq \cdots \subseteq \A^n_{\#}(\XBY)\subseteq \cdots \subseteq \A^0_{\#}(\XBY) = [\XBY,\XBY]^B_B$ is exactly the cardinality of  $I_{(X)}\cup I_{(Y)}$.

\end{proof}

\begin{example}\label{exlkm}
Let $\big(M(G,n)\times K(G,r),p,s\big)$ be an ex-space over $M(G,n)$, where $p\colon M(G,n)\times K(G,r)\to M(G,n)$ is a first coordinate projection. Assume that $r\leq n$.  
From Lemma \ref{rbtsn} we know that $\N\big(M(G,n)\times K(G,r)\big) = \NA\big(K(G,r)\big) = r$. Therefore we have $\A^r_{\#}\big(M(G,n)\times K(G,r)\big) = \AutB\big(M(G,n)\times K(G,r)\big)$. Moreover observe that for all $k<r$, $\pi_k\big(M(G,n)\times K(G,r)\big) = 0$ . Therefore $\A^k_{\#}\big(M(G,n)\times K(G,r)\big) = [M(G,n)\times K(G,r),M(G,n)\times K(G,r)]^{M(G,n)}_{M(G,n)}$ for all $k<r$. Hence we get the strict inclusions 
\begin{myeq}\label{eq4}
\begin{split}
\AutB\big(M(G,n)\times K(G,r)\big)&  =  \A^r_{\#}\big(M(G,n)\times K(G,r)\big) \\
&\subsetneq [M(G,n)\times K(G,r),M(G,n)\times K(G,r)]^{M(G,n)}_{M(G,n)}
\end{split}
\end{myeq}
Hence $\L\big(M(G,n)\times K(G,r)\big) = 1$.
\end{example}

\begin{example}
Let $(X,p,s)$ and $(B\times K(G,m),q,t)$ are two ex-spaces over $B$, where $q\colon B\times K(G,m)\to B$ is a first coordinate projection. Observe that $X\times_B(B\times K(G,m))\cong X\times K(G,m).$ Therefore from Theorem \ref{eslps},
\begin{align*}
\L\big(X\times K(G,m)\big) = &~\L\big(X\times_B(B\times K(G,m)\big)\\
\geq &~\L(X) + \L\big(B\times K(G,m)\big) - {\#}\big(I_X\cap I_{B\times K(G,m)}\big). 
\end{align*}
In particular take $X = M(G,n)\times K(G,r)$ and $B = M(G,n)$, where $m\leq n$ and $r\leq n$ with $m\neq r$. Therefore,
\begin{align*}
\L\big(M(G,n)\times &K(G,r)\times K(G,m)\big) \geq \L\big(M(G,n)\times K(G,r)\big)~ + \\
&\L\big(M(G,n)\times K(G,m)\big) - {\#}\big(I_{M(G,n)\times K(G,r)}\cap I_{M(G,n)\times K(G,m)}\big).
\end{align*}
From Equation \ref{eq4} we have $I_{\big(M(G,n)\times K(G,r)\big)} = \{r-1\}$. Therefore $\L\big(M(G,n)\times K(G,r)\big) = {\#}I_{\big(M(G,n)\times K(G,r)\big)} = 1$.

Similarly we have $I_{\big(M(G,n)\times K(G,m)\big)} = \{m-1\}$ and $\L\big(M(G,n)\times K(G,m)\big) = {\#}I_{\big(M(G,n)\times K(G,m)\big)} = 1$.

\noindent Hence, $${\#}\Big(I_{\big(M(G,n)\times K(G,r)\big)}\cap I_{\big(M(G,n)\times K(G,m)\big)} \Big) = 0 .$$ So we get $$\L\big(M(G,n)\times K(G,r)\times K(G,m)\big) \geq 2.$$

\end{example}
\sect{Higher fibred products}\label{shfb}
In previous sections, we studied several results of the monoid $\A_{\#}^k$, self-closeness number and self length for fibred product of two ex-spaces. In this section we generalised the results for $n$-fold fibred product. 
Let $X_1,\ldots,X_n$ are in $\CWBB$, with fibrations $p_i\colon X_i\to B$ and sections $s_i\colon B\to X_i$ for all $1\leq i\leq n$.  The fibred product of these $n$-spaces is given by 
$$\nprod = X_1\times_B \cdots \times_B X_n =  \{(x_1,\ldots,x_n): p_i(x_i) =p_j(x_j)~~ \forall ~~1\leq i,j\leq n \}.$$ 
For the rest of the section, we will use $\nprod$ to denote the $n$-fold fibred product.
Clearly $\nprod$ is an object in $\CWBB$ where $\pprod \colon \nprod \to B$ is a fibration with section $(s_1,\ldots,s_n)\colon B \to \nprod$ defined as $(\pprod)(x_1,\ldots,x_n) = p_i(x_i)$ and $(s_1,\ldots,s_n)(b) = (s_1(b),\ldots,s_n(b))$. For any $1\leq i\leq n$, we have natural projection map $p_{X_i}: \nprod \to X_i$.

\begin{definition}
A map $f = (f_{X_1},\ldots,f_{X_n}) \in \A^k_{\#}(\nprod)$ is said to be $k$-reducible if $(p_{X_1},\ldots,f_{X_i},\ldots,p_{X_n})\in \A^k_{\#}(\nprod)$, for all $1\leq i\leq n$.
\end{definition}

The following is a generalised version of Proposition \ref{mhom}. 

\begin{proposition}\label{hrc}
Let $f = (f_{X_1},\ldots,f_{X_n})\in \mapbb(\nprod,\nprod)$. Then, for $1\leq i\leq n$, we have $$(p_{X_1},\ldots, f_{X_i},\ldots,p_{X_n})\in \A^k_{\#}(\nprod) \text{ if and only if } f_{X_i}\circ \iota_{X_i}\in \A^k_{\#}(X_i).$$

In particular, $$(p_{X_1},\ldots, f_{X_i},\ldots,p_{X_n})\in \AutB(\nprod) \text{ if and only if } f_{X_i}\circ \iota_{X_i}\in \AutB(X_i).$$
\end{proposition}

\begin{proof}
Let $$\bar{{\bf X}} = X_1\times_B \ldots \times_B X_{i-1}\times_B X_{i+1}\ldots \times_B X_n\times_B X_i.$$ Consider permutation map $\phi:\nprod \to \bar{\bf X}$ defined as, $$(x_{1},\ldots,x_{n}) \mapsto (x_1,\ldots,x_{i-1},x_{i+1},\ldots,  x_n,x_i).$$ 
Clearly $\phi$ is a homeomorphism. 
Therefore $(p_{X_1},\ldots,f_{X_i},\ldots,p_{X_n})\in \A^k_{\#}(\nprod)$ if and only if $\phi \circ (p_{X_1},\ldots,f_{X_i},\ldots,p_{X_n})\circ \phi^{-1}$ =     
$(p'_{X_1},\ldots,p'_{X_{i-1}},p'_{X_{i+1}},\ldots, $ $p'_{X_n},f'_{X_i}) \in \A^k_{\#}(\bar{\bf X}),$
where $p'_{X_j}:\bar{\bf X}\to X_{j}$ are all projection maps for $j \in \{1,\ldots,i-1,i+1,\ldots,n\}$.

\mydiagram[\label{dig4}]{
	\nprod \ar[dd]_{(p_{X_1},\cdots, f_{X_i},\cdots,p_{X_n})}  \ar[rrrr]^{\phi}  &&&&  \bar{\bf X} \ar[dd]^{\phi \circ (p_{X_1},\cdots, f_{X_i},\cdots,p_{X_n})\circ \phi^{-1}} \\\\
	\nprod \ar[rrrr]_{\phi} &&&& \bar{\bf X}
}
We write $\bar{\bf X} = \widetilde{X}\times_B X_i$ where $$\widetilde{X}=X_1\times_B \ldots \times_B X_{i-1}\times_B X_{i+1}\times_B \ldots \times_B X_n. $$
\noindent Then $(p'_{\widetilde{X}},f'_{X_i}) = (p'_{X_1},\ldots,p'_{X_{i-1}},p'_{X_{i+1}},\ldots,p'_{X_n},f'_{X_i})\in \A^k_{\#}(\widetilde{X}\times_B X_i)$ if and only if $f'_{X_i}\circ \iota'_{X_i}\in \A^k_{\#}(X_i)$, where $\iota'_{X_i}\colon X_i\to \bar{\bf X}$ (cf. Proposition \ref{mhom}). Thus $\phi^{-1}\circ (p'_{X_1},\ldots,p'_{X_{i-1}},p'_{X_{i+1}},\ldots,p'_{X_n},f'_{X_i})\circ \phi = (p_{X_1},\ldots,p_{X_{i-1}},f_{X_i},p_{X_{i+1}},\ldots,p_{X_n}) \in \A^k_{\#}(\nprod)$ if and only if  $f_{X_i}\circ \iota_{X_i}\in \A^k_{\#}(X_i),~~~(\text{ since } f'_{X_i}\circ \iota'_{X_i} = f_{X_i}\circ \phi^{-1}\circ \iota'_{X_i} = f_{X_i}\circ \iota_{X_i}).$

\end{proof}
As a corollary of Proposition \ref{hrc}, we have the following. 

\begin{corollary}\label{cnredu}
A map $f \in \A^k_{\#}(\nprod)$ is $k$-reducible if $f_{X_i}\circ \iota_{X_i}\in \A^k_{\#}(X_i)$ for all $1\leq i\leq n$.

In particular, a map $f \in \AutB(\nprod)$ is reducible if $f_{X_i}\circ \iota_{X_i}\in \AutB(X_i)$ for all $1\leq i\leq n$.
\end{corollary}

\begin{proposition}\label{fshe}
Let $X_1,\ldots,X_n\in \CWBB$. For any given $k\geq 0$,  each map in $\A^k_{\#}(\nprod)$ is $k$-reducible. Then a map $f\in \mapbb(\nprod)$ is in $\AutB(\nprod)$ if and only if $f_{X_i}\circ \iota_{X_i}\in \AutB(X_i)$ for all $1\leq i\leq n$.
\end{proposition}

\begin{proof}
Let $f\in \AutB(\nprod)$. So $f\in \A^k_{\#}(\nprod)$,~~ $k\geq 0$. By $k$-reducibility assumption we have $f_{X_i}\circ \iota_{X_i}\in \A^k_{\#}(X_i)$ for all $1\leq i\leq n$ and $k\geq 0$. Therefore for $1\leq i\leq n$ we have $\pi_j(f_{X_i}\circ \iota_{X_i})\colon \pi_j(X_i)\to \pi_j(X_i)$ is isomorphism for all $j$. Then by Whitehead Theorem and Dold Theorem we have $f_{X_i}\circ \iota_{X_i}\in \AutB(X_i)$ for all $1\leq i\leq n$.

Conversely, let $f\in\mapbb(\nprod)$ such that $f_{X_i}\circ \iota_{X_i}\in \AutB(X_i)$ for all $1\leq i\leq n$.  
Thus for $1\leq i\leq n$ we have $f_{X_i}\circ \iota_{X_i}\in \A^k_{\#}(X_i)$ for all $k\geq 0$. From Proposition \ref{hrc}, $(p_{X_1},\ldots,f_{X_i},\ldots,p_{X_n})\in \A^k_{\#}(\nprod)$ for $1\leq i\leq n$ and $\forall k\geq 0$. Therefore for some $b\in B$, we have $(p_{X_1},\ldots,f_{X_i},\ldots,p_{X_n})_b\in \AAA^k_{\#}(\nprod_b)$ for all $k\geq 0$ and $1\leq i\leq n$, using Proposition \ref{iprop}. Hence for $1\leq i\leq n$ we say that $(p_{X_{1b}},\ldots,f_{X_{ib}},\ldots,p_{X_{nb}})\in \Aut(\nprod_b)$ whenever $f_{X_i}\circ \iota_{X_i}\in \AutB(X_i)$ which is equivalent to $f_{X_{ib}}\circ \iota_{X_{ib}} = (f_{X_i}\circ \iota_{X_i})_b\in \Aut(X_{ib})$ from Proposition \ref{iprop}. Observe that,
\begin{align*}
(p_{X_{1b}},\ldots,p_{X_{n-2b}},f_{X_{n-1b}},f_{X_{nb}})=\big[p_{X_{1b}},\ldots,p_{X_{n-1b}},f_{X_{nb}}\circ &(p_{X_{1b}},\ldots,f_{X_{n-1b}},p_{X_{nb}})^{-1}\big]\\ 
&\circ (p_{X_{1b}},\ldots,f_{X_{n-1b}},p_{X_{nb}}).
\end{align*}
From \cite [Proposition 2.3(a)]{pshp} we have,
 $$(p_{X_{1b}},\ldots,f_{X_{ib}},\ldots,p_{X_{nb}})^{-1} = (p_{X_{1b}},\ldots,\bar{f}_{X_{ib}},\ldots,p_{X_{nb}})\in\Aut(\nprod_b),~~\text{ for } 1\leq i\leq n.$$ 
Again,
\begin{align*}
&(p_{X_{1b}},\ldots,f_{X_{nb}})\circ (p_{X_{1b}},\ldots,f_{X_{n-1b}},p_{X_{nb}})^{-1}\\
= &\big(p_{X_{1b}},\ldots,\bar{f}_{X_{n-1b}},f_{X_{nb}}\circ (p_{X_{1b}},\ldots,\bar{f}_{X_{n-1b}},p_{X_{nb}})\big)\in \Aut(\nprod_b).
\end{align*}
By reducibility assumption we have $f_{X_{nb}}\circ (p_{X_{1b}},\ldots,\bar{f}_{X_{n-1b}},p_{X_{nb}})\circ \iota_{X_{nb}}\in \Aut(X_{nb})$. Therefore $\big(p_{X_{1b}},\ldots,f_{X_{nb}}\circ (p_{X_{1b}},\ldots,f_{X_{n-1b}},p_{X_{nb}})^{-1}\big)\in \Aut(\nprod_b)$. Consequently 
$$(p_{X_{1b}},\ldots,p_{X_{n-2b}},f_{X_{n-1b}},f_{X_{nb}})\in \Aut(\nprod_b) \text{ since } f_{X_{ib}}\circ \iota_{X_{ib}}\in \Aut(X_{ib}), \text{for } i = n-1,n.$$ 
Similarly,
\begin{align*}
(p_{X_{1b}},\ldots,p_{X_{n-3b}},&f_{X_{n-2b}},f_{X_{n-1b}},f_{X_{nb}})
= \big[p_{X_{1b}},\ldots,f_{X_{n-1b}}\circ (p_{X_{1b}},\ldots,f_{X_{n-2b}},\ldots,p_{X_{nb}})^{-1},\\
&f_{X_{nb}}\circ (p_{X_{1b}},\ldots,f_{X_{n-2b}},\ldots,p_{X_{nb}})^{-1}\big]\circ (p_{X_{1b}},\ldots,f_{X_{n-2b}},\ldots,p_{X_{nb}}).
\end{align*}
Moreover,
\begin{align*}
&(p_{X_{1b}},\ldots,p_{X_{n-2b}},f_{X_{n-1b}},f_{X_{nb}})\circ (p_{X_{1b}},\ldots,f_{X_{n-2b}},\ldots,p_{X_{nb}})^{-1}
= \big(p_{X_{1b}},\ldots,\bar{f}_{X_{n-2b}},\\ &~~f_{X_{n-1b}}\circ (p_{X_{1b}},\ldots,\bar{f}_{X_{n-2b}},\ldots,p_{X_{nb}}),
f_{X_{nb}}\circ (p_{X_{1b}},\ldots,\bar{f}_{X_{n-2b}},\ldots,p_{X_{nb}})\big)\in \Aut(\nprod_b).
\end{align*}
By reducibility assumption we have $f_{X_{n-1b}}\circ (p_{X_{1b}},\ldots,\bar{f}_{X_{n-2b}},\ldots,p_{X_{nb}})\circ \iota_{X_{n-1b}}\in \Aut(X_{n-1b})$ and $f_{X_{nb}}\circ(p_{X_{1b}},\ldots,\bar{f}_{X_{n-2b}},\ldots,p_{X_{nb}})\circ \iota_{X_{nb}}\in \Aut(X_{nb})$. Therefore $\big(p_{X_{1b}},\ldots,f_{X_{n-1b}}\circ (p_{X_{1b}},\ldots,f_{X_{n-2b}},\ldots,p_{X_{nb}})^{-1},f_{X_{nb}}\circ (p_{X_{1b}},\ldots,f_{X_{n-2b}},\ldots,p_{X_{nb}})^{-1}\big)\in \Aut(\nprod_b)$, using the previous step. Consequently,
 $$(p_{X_{1b}},\ldots,p_{X_{n-3b}},f_{X_{n-2b}},f_{X_{n-1b}},f_{X_{nb}})\in \Aut(\nprod_b) \text{ since } f_{X_{ib}}\circ \iota_{X_{ib}}\in \Aut(X_{ib}), \text{ for } i = $$ 
$n-2,n-1,n$. Proceeding similarly we finally obtain $f_b\in \Aut(\nprod_b)$.
Therefore $f\in \AutB(\nprod)$ from Proposition \ref{iprop}.

\end{proof}

Let $\prod_{\hat{i}}:=X_{1}\times_B \ldots \times_B \hat{X_{i}}\times_B \ldots \times_B X_{n}$, i.e. $\prod_{\hat{i}}$ denotes the subproduct of $\nprod$ obtained by omitting $X_{i}$. As before, we define some sub-monoids of  $\A^k_{\#}(\nprod)$ as:
$$\A^k_{\#,\prod_{\hat i}}(\nprod) := \big\{f\in \A^k_{\#}(\nprod): p_{X_j}\circ f = p_{X_j}|~j\neq i,~1\leq j\leq n\big\}.$$ 
Clearly, the elements of $\A^k_{\#,\prod_{\hat i}}(\nprod)$ are of the form $(p_{X_1},\ldots,f_{X_i},\ldots,p_{X_n})$ where $f_{X_i}\circ \iota_{X_i}\in \A^k_{\#}(X_i)$ and this a monoid. Further define,
$$\A^{k,X_i}_{\#,\prod_{\hat i}}(\nprod): = \big\{(p_{X_1},\ldots.f_{X_i},\ldots,p_{X_n})\in \A^k_{\#,\prod_{\hat i}}(\nprod):~ (p_{X_1},\ldots.f_{X_i},\ldots,p_{X_n})\circ \iota_{X_i} = \iota_{X_i} \big\}.$$ 
 The following theorem gives an explicit relation of $\A^k_{\#}(\nprod)$ with these submonoids, generalising Theorem \ref{ftsg}.

\begin{theorem}\label{mthnps}
Assume that each map $f \in \A^k_{\#}(\nprod)$ is $k$-reducible and satisfies the condition $f_{X_i} \simeq^B_B f_{X_i}\circ \iota_{X_i}\circ p_{X_i}, ~~ \forall~2\leq i\leq n$. Then $$\A^k_{\#}(\nprod) = \A^k_{\#,\prod_{\hat n}}(\nprod)\cdots \A^k_{\#,\prod_{\hat 1}}(\nprod).$$
\end{theorem}

\begin{proof}
Let $f \in \A^k_{\#}(\nprod)$, by $k$-reducibility $(p_{X_1},\ldots,f_{X_i},\ldots,p_{X_n})\in \A^k_{\#,\prod_{\hat i}}(\nprod),~~~\text{ for } 1\leq i\leq n$. Since, $$\A^k_{\#,\prod_{\hat n}}(\nprod)\cap \cdots \cap \A^k_{\#,\prod_{\hat 1}}(\nprod) = (p_{X_1},\ldots,p_{X_n}).$$ Therefore if the factorization exists, it will be unique. Observe that,

\begin{align*}
&(p_{X_1},\ldots,p_{X_{n-1}},f_{X_n})\circ (p_{X_1},\ldots,f_{X_{n-1}},p_{X_n})\circ \cdots \circ(f_{X_1},p_{X_2},\ldots,p_{X_n})\\ 
=~&(p_{X_1},\ldots,p_{X_{n-1}},f_{X_n})\circ (p_{X_1},\ldots,f_{X_{n-1}},p_{X_n})\circ \cdots \circ\\       ~&~\big(f_{X_1},f_{X_2}\circ(f_{X_1},p_{X_2},\ldots,p_{X_n}),\ldots,p_{X_n} \big)\\
\simeq^B_B&(p_{X_1},\ldots,p_{X_{n-1}},f_{X_n})\circ (p_{X_1},\ldots,f_{X_{n-1}},p_{X_n})\circ \cdots \circ\\ ~&~\big(f_{X_1},f_{X_2}\circ \iota_{X_2}\circ p_{X_2}\circ(f_{X_1},p_{X_2},\ldots,p_{X_n}),\ldots,p_{X_n}\big)\\
=~&(p_{X_1},\ldots,p_{X_{n-1}},f_{X_n})\circ (p_{X_1},\ldots,f_{X_{n-1}},p_{X_n})\circ \cdots \circ(f_{X_1},f_{X_2}\circ \iota_{X_2}\circ p_{X_2},\cdots,p_{X_n})\\
\simeq^B_B&(p_{X_1},\ldots,p_{X_{n-1}},f_{X_n})\circ (p_{X_1},\ldots,f_{X_{n-1}},p_{X_n})\circ \cdots \circ(f_{X_1},f_{X_2},p_{X_3},\ldots,p_{X_n})\\
&~\vdots \\
=~&(p_{X_1},\ldots,p_{X_{n-1}},f_{X_n})\circ (f_{X_1},\ldots,f_{X_{n-1}},p_{X_n})\\
=~&\big(f_{X_1},\ldots,f_{X_{n-1}},f_{X_n}\circ (f_{X_1},\ldots,f_{X_{n-1}},p_{X_n})\big)\\
\simeq^B_B &\big(f_{X_1},\ldots,f_{X_{n-1}},f_{X_n}\circ \iota_{X_n}\circ p_{X_n}\circ (f_{X_1},\ldots,f_{X_{n-1}},p_{X_n})\big)\\
=~&\big(f_{X_1},\ldots,f_{X_{n-1}},f_{X_n}\circ \iota_{X_n}\circ p_{X_n} \big)\\
\simeq^B_B &\big(f_{X_1},\ldots,f_{X_{n-1}},f_{X_n} \big)
\end{align*}
Therefore we get the desired result.

\end{proof}

\begin{proposition}\label{splthp}
Let $X_1,\ldots,X_n\in \CWBB$. For each fixed $i\in \{1,\ldots,n\}$, we have a monoid homomorphism $\phi_{X_i}\colon \A^k_{\#,\prod_{\hat i}}(\nprod)\to \A^k_{\#}(X_i)$ defined as $$\phi_{X_i}\big(p_{X_1},\ldots,f_{X_i},\cdots,p_{X_n}\big) = f_{X_i}\circ \iota_{X_i}.$$ 
Moreover there exists split exact sequences, $$0\to \A^{k,X_i}_{\#,\prod_{\hat i}}(\nprod)\to \A^k_{\#,\prod_{\hat i}}(\nprod)\xrightarrow[]{\phi_{X_i}} \A^k_{\#}(X_i)\to 0.$$
\end{proposition}

\begin{proof}
From Proposition \ref{hrc}, we say that $\phi_{X_i}$ is well-defined. Let us observe that for any self-map $g\colon X_i\to X_i$, we have $p_{X_j}\circ \iota_{X_i}\circ g = p_{X_j}\circ \iota_{X_i}$, for $1\leq j(\neq i)\leq n$.
Moreover for any $1\leq i\leq n$,
\begin{align*}
\iota_{X_i}\circ f_{X_i}\circ \iota_{X_i} = &~\big(p_{X_1}\circ \iota_{X_i}\circ f_{X_i}\circ \iota_{X_i},\ldots,p_{X_i}\circ \iota_{X_i}\circ f_{X_i}\circ \iota_{X_i},\ldots,p_{X_n}\circ \iota_{X_i}\circ f_{X_i}\circ \iota_{X_i} \big)\\
= &~\big(p_{X_1}\circ \iota_{X_i},\ldots,f_{X_i}\circ \iota_{X_i},\ldots,p_{X_n}\circ \iota_{X_i} \big)\\
= &~\big(p_{X_1},\ldots,f_{X_i},\ldots,p_{X_n}\big)\circ \iota_{X_i}.
\end{align*}
Therefore, 
\begin{align*}
&\phi_{X_i}\big((p_{X_1},\ldots,f_{X_i},\ldots,p_{X_n})\circ (p_{X_1},\ldots,g_{X_i},\ldots,p_{X_n})\big)\\
=~&\phi_{X_i}\big((p_{X_1},\ldots,f_{X_i}\circ (p_{X_1},\ldots,g_{X_i},\ldots,p_{X_n}),\ldots,p_{X_n})\big)\\
=~& f_{X_i}\circ (p_{X_1},\ldots,g_{X_i},\ldots,p_{X_n})\circ \iota_{X_i}\\
=~& f_{X_i}\circ \iota_{X_i}\circ g_{X_i}\circ \iota_{X_i}\\
=~& \phi_{X_i}\big(p_{X_1},\ldots,f_{X_i},\ldots,p_{X_n}\big)\circ \phi_{X_i}\big(p_{X_1},\ldots,g_{X_i},\ldots,p_{X_n}\big).
\end{align*}
\noindent Hence $\phi_{X_i}$ is a monoid homomorphism. As in the first part of Theorem \ref{ftsg}, we can show that $\Ker(\phi_{X_i}) = \A^{k,X_i}_{\#,\prod_{\hat i}}(\nprod)$. Consider the map $\psi_{X_i}\colon \A^k_{\#}(X_i)\to \A^k_{\#,\prod_{\hat i}}(\nprod)$ defined as $$h\mapsto \big(p_{X_1},\ldots,h\circ p_{X_i},\ldots,p_{X_n}\big).$$ This map is well-defined by Proposition \ref{hrc}. Clearly $\psi_{X_i}$ is a monoid homomorphism such that $\phi_{X_i}\circ \psi_{X_i} = \Id_{\A^k_{\#}(X_i)}.$

\end{proof}

The following theorem is a generalization of Theorem \ref{hbthm} for $n$-fibred product. 
\begin{theorem}\label{hbfn}
Let us assume that each map $f \in \A^k_{\#}(\nprod)$ is $k$-reducible and satisfying the conditions $f_{X_i} \simeq^B_B f_{X_i}\circ \iota_{X_i}\circ p_{X_i}, ~~ \forall~~2\leq i\leq n$. Then there exits a split short exact sequence $$0\to \A^{k,X_1}_{\#,\prod_{\hat 1}}(\nprod)\to \A^k_{\#}(\nprod)\xrightarrow[]{\phi} \A^k_{\#}(X_1)\times \cdots \times \A^k_{\#}(X_n)\to 0,$$ 
where $\phi(f) = (f_{X_1}\circ \iota_{X_1},\ldots,f_{X_n}\circ \iota_{X_n})$ for each $f\in \A^k_{\#}(\nprod)$.
\end{theorem}

\begin{proof}
Since each  map $f \in \A^k_{\#}(\nprod)$ is $k$-reducible, therefore the map $\phi$ is well-defined by Corollary \ref{cnredu}. For $1\leq j(\neq i)\leq n$, let us observe that $h\circ p_{X_j}\circ \iota_{X_i} = p_{X_j}\circ \iota_{X_i}$ where $h\in \mapbb (X_j)$. Moreover as in the proof of Proposition \ref{splthp} we have $\iota_{X_i}\circ f_{X_i}\circ \iota_{X_i} = \big(p_{X_1},\ldots,f_{X_i},\ldots,p_{X_n}\big)\circ \iota_{X_i}$, for $1\leq i\leq n$.

\noindent Therefore
\begin{align*}
&\phi\big((f_{X_1},\ldots,f_{X_n})\circ (g_{X_1},\ldots,g_{X_n})\big)\\
=~&\phi \big(f_{X_1}\circ (g_{X_1},\ldots,g_{X_n}),\ldots,f_{X_n}\circ (g_{X_1},\ldots,g_{X_n})\big)\\
=~&\big(f_{X_1}\circ (g_{X_1},\ldots,g_{X_n})\circ \iota_{X_1},f_{X_2}\circ (g_{X_1},\ldots,g_{X_n})\circ \iota_{X_2}, \ldots,\\ &~f_{X_n}\circ (g_{X_1},\ldots,g_{X_n})\circ \iota_{X_n}\big)\\ 
\simeq^B_B&\big(f_{X_1}\circ (g_{X_1},g_{X_2}\circ \iota_{X_2}\circ p_{X_2}\ldots,g_{X_n}\circ \iota_{X_n}\circ p_{X_n})\circ \iota_{X_1},\\ &~f_{X_2}\circ \iota_{X_2}\circ p_{X_2}\circ (g_{X_1},\ldots,g_{X_n})\circ \iota_{X_2},\ldots,f_{X_n}\circ \iota_{X_n}\circ p_{X_n}\circ (g_{X_1},\ldots,g_{X_n})\circ \iota_{X_n}\big)\\
=~&\big(f_{X_1}\circ (g_{X_1}\circ \iota_{X_1},g_{X_2}\circ \iota_{X_2}\circ p_{X_2}\circ \iota_{X_1}\ldots,g_{X_n}\circ \iota_{X_n}\circ p_{X_n}\circ \iota_{X_1}),\\ &~f_{X_2}\circ \iota_{X_2}\circ g_{X_2}\circ \iota_{X_2},\ldots,f_{X_n}\circ \iota_{X_n}\circ g_{X_n}\circ \iota_{X_n}\big)\\
=~&\big(f_{X_1}\circ (g_{X_1}\circ \iota_{X_1},p_{X_2}\circ \iota_{X_1},\ldots,p_{X_n}\circ \iota_{X_1}),f_{X_2}\circ \iota_{X_2}\circ g_{X_2}\circ \iota_{X_2},\\ &~\ldots,f_{X_n}\circ \iota_{X_n}\circ g_{X_n}\circ \iota_{X_n}\big)\\
=~&\big(f_{X_1}\circ (g_{X_1},p_{X_2},\ldots,p_{X_n})\circ \iota_{X_1},f_{X_2}\circ \iota_{X_2}\circ g_{X_2}\circ \iota_{X_2},\ldots,f_{X_n}\circ \iota_{X_n}\circ g_{X_n}\circ \iota_{X_n}\big)\\
=~&\big(f_{X_1}\circ \iota_{X_1}\circ g_{X_1}\circ \iota_{X_1},\ldots,f_{X_n}\circ \iota_{X_n}\circ g_{X_n}\circ \iota_{X_n}\big)\\
=~&\phi\big((f_{X_1},\ldots,f_{X_n})\big)\circ \phi\big((g_{X_1},\ldots,g_{X_n})\big)
\end{align*}
This shows that $\phi$ is a monoid homomorphism. 

Any element of $\A^{k,X_i}_{\#,\prod_{\hat i}}(\nprod)$ is of the form $(p_{X_1},\ldots,f_{X_i},\ldots,p_{X_n})$ satisfying the condition $(p_{X_1},\ldots,f_{X_i},\ldots,p_{X_n})\circ \iota_{X_i} = \iota_{X_i}.$
First pre-composing with $p_{X_i}$ and then post-composing with $p_{X_i}$ we have $f_{X_i}\circ \iota_{X_i}\circ p_{X_i} = p_{X_i}$. By assumption, $f_{X_i}\simeq^B_B f_{X_i}\circ \iota_{X_i}\circ p_{X_i} =  p_{X_i}$, for all $2\leq i\leq n$. Therefore $\A^{k,X_i}_{\#,\prod_{\hat i}}(\nprod) = \{\ast\}$ for all $2\leq i\leq n$. Hence $\Ker(\phi)$ can be identified with $\Ker(\phi_{X_1}) = \A^{k,X_1}_{\#,\prod_{\hat 1}}(\nprod)$. Further we consider a map $\psi\colon \A^k_{\#}(X_1)\times \cdots \times \A^k_{\#}(X_n) \to \A^k_{\#}(\nprod)$ defined as $$(f_1,\ldots,f_n)\mapsto (f_1\circ p_{X_1},\ldots,f_n\circ p_{X_n}).$$
Observe that, $$(f_1\circ p_{X_1},\ldots,f_n\circ p_{X_n}) = (f_1\circ p_{X_1},\ldots,p_{X_n})\circ \cdots\circ (p_{X_1},\ldots,f_n\circ p_{X_n})\in \A^k_{\#}(\nprod),$$
since $(p_{X_1},\ldots,f_{i}\circ p_{X_i},\ldots,p_{X_n})\in \A^k_{\#}(\nprod)$ for $1\leq i\leq n$, from Proposition \ref{hrc}. Therefore $\psi$ is well-defined. Clearly $\psi$ is a monoid homomorphism  such that $\phi \circ \psi = \Id_{\A^k_{\#}(X_1)\times \cdots \times \A^k_{\#}(X_n)}$.

\end{proof}

We have the following generalisation of Theorem \ref{scn} and Theorem \ref{escn}. 

\begin{theorem}\label{escnp}
For $X_1,\ldots,X_n\in \CWBB$, we have $$\N(\nprod) \geq \max \big \{\N(X_1),\ldots,\N(X_n) \big \}.$$
Further, for any $k\geq 0$ if each map in $\A^k_{\#}(\nprod)$ is $k$-reducible, then 
$$\N(\nprod) = \max \big \{\N(X_1),\ldots,\N(X_n) \big \}.$$
\end{theorem}

\begin{proof}
From Theorem \ref{scn} we have,
$$\N(\nprod) \geq \max \big\{\N(X_1\times_B \cdots \times_B X_{n-1}),\N(X_n)\big\}.$$
Similarly,  
$$\N(X_1\times_B \cdots \times_B X_{n-1}) \geq \max \big\{\N(X_1\times_B \cdots \times_B X_{n-2}),\N(X_{n-1})\big\}.$$ 
Combing this two inequalities we get, 
$$\N(\nprod) \geq \max \big\{\N(X_1\times_B \cdots \times_B X_{n-2}),\N(X_{n-1}),\N(X_n)\big\}.$$
Continuing further, we obtain the desired inequality. 

For the second part, consider the case $\N(X_i)<\infty$ for all $1\leq i\leq n$. Let $k =\max \big \{\N(X_1),\ldots,\N(X_n) \}$ and $f\in \A^k_{\#}(\nprod)$. By $k$-reducibility we have $f_{X_i}\circ \iota_{X_i}\in \A^k_{\#}(X_i)  = \AutB(X_i)$. Therefore $f\in \AutB(\nprod)$, using Proposition \ref{fshe}. Thus $\A^k_{\#}(\nprod) = \AutB(\nprod)$. So $\N(\nprod) \leq k.$ Hence by the first part, we obtain $$\N(\nprod) = \max \big \{\N(X_1),\ldots,\N(X_n) \big \}.$$

\end{proof}

\begin{proposition}\label{slnps}
Let $X_1,\ldots,X_n\in \CWBB$. For any non-negative integer $l$, we have
\begin{enumerate}[(a)]
\item
If $\A^{l+1}_{\#}(X_i)\subsetneq \A^l_{\#}(X_i)$ for some $i\in \{1,\ldots,n\}$ then $\A^{l+1}_{\#}(\nprod)\subsetneq \A^l_{\#}(\nprod)$.

\item The converse holds if given any $k\geq 0$, each map in $\A^k_{\#}(\nprod)$ is $k$-reducible and for any $2\leq i\leq n$ satisfies the conditions $f_{X_i} \simeq^B_B f_{X_i}\circ \iota_{X_i}\circ p_{X_i}$, for all $f\in \A^k_{\#}(\nprod)$.
\end{enumerate}
\end{proposition}

\begin{proof}

(a)
For some fixed $i$, let $\A^{l+1}_{\#}(X_i)\subsetneq \A^l_{\#}(X_i)$. If possible assume that $\A^{l+1}_{\#}(\nprod) = \A^l_{\#}(\nprod)$. Let $f\in \A^l_{\#}(X_i) \setminus \A^{l+1}_{\#}(X_i)$. Then for some $g\in \A^l_{\#}(\prod_{\hat i})$ we have $g\times_Bf\in \A^l(\prod_{\hat i}\times_B X_i)$. As in the proof of`Proposition \ref{hrc}, we have $\phi$ such that $\phi^{-1}\circ g\times_Bf\circ \phi \in \A^l_{\#}(\nprod)$ (see Diagram \ref{dig4}). Thus $\phi^{-1}\circ g\times_Bf\circ \phi \in \A^{l+1}_{\#}(\nprod) = \A^l_{\#}(\nprod)$. So $\pi_i(\phi^{-1}\circ g\times_Bf\circ \phi)$ is an isomorphism for all $i\leq l+1$. Therefore using the diagram \ref{dig4} we have $\pi_i(g\times_B f)$ is isomorphism for all $i\leq l+1$, i.e. $g\times_B f\in \A^{l+1}_{\#}(\prod_{\hat i}\times_B X_i)$. As in the proof of Proposition \ref{slps}(a),  we have $f\in \A^{l+1}_{\#}(X_i)$. This is a contradiction.

(b)
For the converse implication it is given that $\A^{l+1}_{\#}(\nprod)\subsetneq \A^l_{\#}(\nprod)$. If possible assume that $\A^{l+1}_{\#}(X_i) = \A^l_{\#}(X_i)$ for all $1\leq i\leq n$. Let $f \in \A^l_{\#}(\nprod) \setminus \A^{l+1}_{\#}(\nprod)$. Thus by $k$-reducibility we have $f_{X_i}\circ \iota_{X_i}\in \A^l_{\#}(X_i)$, so that $f_{X_i}\circ \iota_{X_i}\in \A^{l+1}_{\#}(X_i) = \A^l_{\#}(X_i)$ for all $1\leq i\leq n$. From Proposition \ref{hrc} we have $(p_{X_1},\ldots, f_{X_i},\ldots,p_{X_n})\in \A^{l+1}_{\#}(\nprod)$, for all $1\leq i\leq n$. Further, as in the proof of Theorem \ref{mthnps},
$$f = (p_{X_1},\ldots,p_{X_{n-1}},f_{X_n})\circ (p_{X_1},\ldots,f_{X_{n-1}},p_{X_n})\circ \cdots \circ(f_{X_1},p_{X_2},\ldots,p_{X_n}) \in \A^{l+1}_{\#}(\nprod).$$ 
 This contradicts the fact that $f \in \A^l_{\#}(\nprod) \setminus \A^{l+1}_{\#}(\nprod)$. Hence $\A^{l+1}_{\#}(X_i)\subsetneq \A^l_{\#}(X_i)$ for some $i\in \{1,\ldots,n\}$.

\end{proof}	

We now obtain a generalisation of Theorem \ref{eslps} for $n$-fibred product. 

\begin{theorem}\label{slonps}
For $X_1,\ldots,X_n\in \CWBB$, we have $$\sum_{i = 1}^n \L(X_i) - \mathcal{C} \leq \L(\nprod),$$ 
$$ \text{ where } \mathcal{C} = \sum_{j = 2}^n (-1)^j \sum_{1\leq i_1< \cdots <i_j\leq n} \#(I_{(X_{i_1})}\cap \cdots \cap I_{(X_{i_j})}).$$
Moreover the equality holds if given any $k\geq 0$, each map in $\A^k_{\#}(\nprod)$ is $k$-reducible and for any $2\leq i\leq n$ satisfies the conditions $f_{X_i} \simeq^B_B f_{X_i}\circ \iota_{X_i}\circ p_{X_i}$, for all $f\in \A^k_{\#}(\nprod)$.
\end{theorem}	

\begin{proof}
Note that if any one of $\L(X_i)$ is infinite then $\L(\nprod)$ is infinite. So we only consider the case $\L(X_i)<\infty$ for all $1\leq i\leq n$. From Proposition \ref{slnps}(a) we have, $l\in \cup_{i = 1}^n I_{(X_i)}$ implies $\A^{l+1}_{\#}(\nprod)\subsetneq \A^l_{\#}(\nprod)$. Therefore the number of strict inclusions of the chain $\AutB({\bf X})\subseteq \cdots \subseteq \A^m_{\#}({\bf X})\subseteq \cdots \subseteq \A^0_{\#}({\bf X}) = [{\bf X},{\bf X}]_B^B$ is at least cardinality of $\cup_{i = 1}^n I_{(X_i)}$. Observe that,

\begin{align*}
 \#(\cup_{i = 1}^n I_{(X_i)}) =& \sum_{j = 1}^n (-1)^{j+1} \sum_{1\leq i_1< \cdots <i_j\leq n} \#(I_{(X_{i_1})}\cap \cdots \cap I_{(X_{i_j})})\\
 =& \sum_{i = 1}^n \#(I_{(X_i)}) - \sum_{j = 2}^n (-1)^j \sum_{1\leq i_1< \cdots <i_j\leq n} \#(I_{(X_{i_1})}\cap \cdots \cap I_{(X_{i_j})})\\
 =& \sum_{i = 1}^n \L(X_i) - \mathcal{C}
\end{align*}
Therefore we get the desired result.

For the second part of the proof, note that $l\in \cup_{i = 1}^n I_{(X_i)}$ if and only if $\A^{l+1}_{\#}(\nprod)\subsetneq \A^l_{\#}(\nprod)$, by Proposition \ref{slnps}(b). Therefore the number of strict inclusions of the chain $\AutB({\bf X})\subseteq \cdots \subseteq \A^m_{\#}({\bf X})\subseteq \cdots \subseteq \A^0_{\#}({\bf X}) = [{\bf X},{\bf X}]_B^B$ is exactly the cardinality of $\cup_{i = 1}^n  I_{(X_i)}$.

\end{proof}
\begin{corollary}\label{rslscn}
Let $X_1,\ldots,X_n\in \CWBB$. For any $k\geq 0$, if each map in $\A^k_{\#}(\nprod)$ is $k$-reducible, then 
$$\sum_{i = 1}^n \L(X_i) - \mathcal{C} \leq \L(\nprod) \leq \max \big \{\N(X_1),\ldots,\N(X_n) \big \}.$$
\end{corollary}

\begin{proof}
The left hand side inequality immediately follows from above Theorem \ref{slonps}. From Proposition \ref{rbslsc} we have $\L({\bf X})\leq \N({\bf X})$. Moreover for any $k\geq 0$, each map in $\A^k_{\#}(\nprod)$ is $k$-reducible, therefore $\N({\bf X}) = \max \big \{\N(X_1),\ldots,\N(X_n) \big \}$ from Theorem \ref{escnp}. Hence the desired result.

\end{proof}

\begin{example}
Consider the ex-spaces $\big(B\times K(G,m_i),p_i,s_i\big)$ over $B$ with $p_i$ first coordinate projection. Take $B = M(G,n)$ with the assumptions $m_i\leq n$ and $m_i\neq m_j$ for all $1\leq i<j\leq n$. Observe that,
$$\big(B\times K(G,m_1)\big)\times_B \cdots \times_B \big(B\times K(G,m_n)\big)\cong B\times K(G,m_1)\cdots\times K(G,m_n),$$ 
where the right hand side is an ex-space over $B$ with first factor projection. For any given $k\geq 0$,
each map in $\AAA^k_{\#}\big(K(G,m_1)\times \cdots \times K(G,m_n)\big)$ is $k$-reducible, so using Corollary \ref{kred} we say that each map in $\A^k_{\#}\big(B\times K(G,m_1)\times \cdots \times K(G,m_n)\big)$ is $k$-reducible. From Lemma \ref{rbtsn} we have,
$$\N\big(M(G,n)\times K(G,m_i)\big) = \NA\big(K(G,m_i)\big) = m_i, ~~1\leq i\leq n.$$
 Moreover from Equation \ref{eq4} we have $I_{\big(M(G,n)\times K(G,m_i)\big)} = \{m_i-1\}$ for $1\leq i\leq n$. Therefore $\mathcal{C} = 0$ and so from Corollary \ref{rslscn} we have  
 $$n\leq \L\big(M(G,n)\times K(G,m_1)\times \cdots \times K(G,m_n)\big)\leq \max \{m_1,\ldots,m_n\}.$$ 
Hence  $$\L\big(M(G,n)\times K(G,m_1)\times \cdots \times K(G,m_n)\big) = n.$$
\end{example}

\end{document}